\documentclass[11pt]{amsart}
\usepackage{amssymb}
\usepackage{amsxtra}
\usepackage{graphics}
\usepackage{latexsym}
\usepackage{enumerate}
\usepackage{xcolor}
\usepackage{amsmath}
\usepackage{amssymb,amsthm,amsfonts}
\usepackage{mathtools}
\usepackage{amscd}
\usepackage[arrow, matrix, curve]{xy}
\usepackage[mathscr]{euscript}
\usepackage{syntonly}

\title[Intersection de Rham complexes in positive characteristic]{Intersection de Rham complexes in positive characteristic}

\author[Mao Sheng]{Mao Sheng}
\email{msheng@ustc.edu.cn}
\address{School of Mathematical Sciences,
	University of Science and Technology of China, Hefei, 230026, China}

\author[Zebao Zhang]{Zebao Zhang}
\email{zhangzebao@bicmr.pku.edu.cn}
\address{Peking University, Beijing International Center for Mathematical Research,
	Beijing, 100871, China}

\begin{document}
	%%%%%%%%%%%%%%%%%%%% Text italic %%%%%%%%%%%%%%%%%%%%%%%%%%%%
	\theoremstyle{plain}
	\newtheorem{thm}{Theorem}[section]
	\newtheorem{theorem}[thm]{Theorem}
	\newtheorem*{theorem*}{Theorem}
	\newtheorem*{theoremA*}{Theorem A}
	\newtheorem*{theoremB*}{Theorem B}
	\newtheorem*{theoremC*}{Theorem C}
	\newtheorem*{definition*}{Definition}
	\newtheorem{lemma}[thm]{Lemma}
	\newtheorem{sublemma}[thm]{Sublemma}
	\newtheorem{corollary}[thm]{Corollary}
	\newtheorem*{corollary*}{Corollary}
	\newtheorem{proposition}[thm]{Proposition}
	\newtheorem{addendum}[thm]{Addendum}
	\newtheorem{variant}[thm]{Variant}
	%%%%%%%%%%%%%%%%%%%% Text roman %%%%%%%%%%%%%%%%%%%%%%%%%%%%%
	\theoremstyle{definition}
	\newtheorem{lemma-definition}[thm]{Lemma-Definition}
	\newtheorem{proposition-definition}[thm]{Proposition-Definition}
	\newtheorem{theorem-definition}[thm]{Theorem-Definition}
	\newtheorem{construction}[thm]{Construction}
	\newtheorem{notations}[thm]{Notations}
	\newtheorem{question}[thm]{Question}
	\newtheorem{problem}[thm]{Problem}
	\newtheorem*{problem*}{Problem}
	\newtheorem{remark}[thm]{Remark}
	\newtheorem*{remark*}{Remark}
	\newtheorem{remarks}[thm]{Remarks}
	\newtheorem{definition}[thm]{Definition}
	\newtheorem{claim}[thm]{Claim}
	\newtheorem{assumption}[thm]{Assumption}
	\newtheorem{assumptions}[thm]{Assumptions}
	\newtheorem{properties}[thm]{Properties}
	\newtheorem{example}[thm]{Example}
	\newtheorem{conjecture}[thm]{Conjecture}
	\numberwithin{equation}{thm}
	
	% Skriptbuchstaben
	\newcommand{\sA}{{\mathcal A}}
	\newcommand{\sB}{{\mathcal B}}
	\newcommand{\sC}{{\mathcal C}}
	\newcommand{\sD}{{\mathcal D}}
	\newcommand{\sE}{{\mathcal E}}
	\newcommand{\sF}{{\mathcal F}}
	\newcommand{\sG}{{\mathcal G}}
	\newcommand{\sH}{{\mathcal H}}
	\newcommand{\sI}{{\mathcal I}}
	\newcommand{\sJ}{{\mathcal J}}
	\newcommand{\sK}{{\mathcal K}}
	\newcommand{\sL}{{\mathcal L}}
	\newcommand{\sM}{{\mathcal M}}
	\newcommand{\sN}{{\mathcal N}}
	\newcommand{\sO}{{\mathcal O}}
	\newcommand{\sP}{{\mathcal P}}
	\newcommand{\sQ}{{\mathcal Q}}
	\newcommand{\sR}{{\mathcal R}}
	\newcommand{\sS}{{\mathcal S}}
	\newcommand{\sT}{{\mathcal T}}
	\newcommand{\sU}{{\mathcal U}}
	\newcommand{\sV}{{\mathcal V}}
	\newcommand{\sW}{{\mathcal W}}
	\newcommand{\sX}{{\mathcal X}}
	\newcommand{\sY}{{\mathcal Y}}
	\newcommand{\sZ}{{\mathcal Z}}
	% Sonderbuchstaben mit Doppellinie
	\newcommand{\A}{{\mathbb A}}
	\newcommand{\B}{{\mathbb B}}
	\newcommand{\C}{{\mathbb C}}
	\newcommand{\D}{{\mathbb D}}
	\newcommand{\E}{{\mathbb E}}
	\newcommand{\F}{{\mathbb F}}
	\newcommand{\G}{{\mathbb G}}
	\newcommand{\HH}{{\mathbb H}}
	\newcommand{\I}{{\mathbb I}}
	\newcommand{\J}{{\mathbb J}}
	\renewcommand{\L}{{\mathbb L}}
	\newcommand{\M}{{\mathbb M}}
	\newcommand{\N}{{\mathbb N}}
	\renewcommand{\P}{{\mathbb P}}
	\newcommand{\Q}{{\mathbb Q}}
	\newcommand{\R}{{\mathbb R}}
	\newcommand{\SSS}{{\mathbb S}}
	\newcommand{\T}{{\mathbb T}}
	\newcommand{\U}{{\mathbb U}}
	\newcommand{\V}{{\mathbb V}}
	\newcommand{\W}{{\mathbb W}}
	\newcommand{\X}{{\mathbb X}}
	\newcommand{\Y}{{\mathbb Y}}
	\newcommand{\Z}{{\mathbb Z}}
	\newcommand{\id}{{\rm id}}
	\newcommand{\rank}{{\rm rank}}
	\newcommand{\END}{{\mathbb E}{\rm nd}}
	\newcommand{\End}{{\rm End}}
	\newcommand{\Hom}{{\rm Hom}}
	\newcommand{\Hg}{{\rm Hg}}
	\newcommand{\tr}{{\rm tr}}
	\newcommand{\Sl}{{\rm Sl}}
	\newcommand{\Gl}{{\rm Gl}}
	\newcommand{\Cor}{{\rm Cor}}
	\newcommand{\Aut}{\mathrm{Aut}}
	\newcommand{\Sym}{\mathrm{Sym}}
	\newcommand{\ModuliCY}{\mathfrak{M}_{CY}}
	\newcommand{\HyperCY}{\mathfrak{H}_{CY}}
	\newcommand{\ModuliAR}{\mathfrak{M}_{AR}}
	\newcommand{\Modulione}{\mathfrak{M}_{1,n+3}}
	\newcommand{\Modulin}{\mathfrak{M}_{n,n+3}}
	\newcommand{\Gal}{\mathrm{Gal}}
	\newcommand{\Spec}{\mathrm{Spec}}
	\newcommand{\res}{\mathrm{res}}
	\newcommand{\coker}{\mathrm{coker}}
	\newcommand{\Jac}{\mathrm{Jac}}
	\newcommand{\HIG}{\mathrm{HIG}}
	\newcommand{\MIC}{\mathrm{MIC}}
	
	\thanks{The work is supported by National Key Research and Development Project SQ2020YFA070080, CAS Project for Young Scientists in Basic Research Grant No. YSBR-032, National Natural Science Foundation of China (Grant No. 11721101), Fundamental Research Funds for the Central Universities.}
	\maketitle
	\centerline{{\itshape Dedicated to Professor Shing-Tung Yau on the occasion of his 70th birthday}}

	\begin{abstract}
		We establish a positive characteristic analogue of intersection cohomology theory for variations of Hodge structure. It includes: a) the de Rham-Higgs comparison theorem for the intersection de Rham complex; b) the $E_1$-degeneration theorem for the intersection de Rham complex of a periodic de Rham bundle; c) the Kodaira-Saito vanishing theorem for the intersection cohomology groups of a periodic Higgs bundle. These results generalize the decomposition theorem of Deligne-Illusie \cite{DI} and the de Rham-Higgs theorem of Ogus-Vologodsky \cite{OV}, the $E_1$-degneration theorem of Deligne-Illusie \cite{DI}, Illusie \cite{IL90}, Faltings \cite{Fa} and the Kodaira-Saito vanishing theorem of Arapura \cite{Arapura}. As an application, we give an algebraic proof of the $E_1$-degeneration theorem due to Cattani-Kaplan-Schmid \cite{CKS} and Kashiwara-Kawai \cite{KK}, and the vanishing theorem of Saito \cite{Saito} for VHSs of geometric origin.
	\end{abstract}
	
	\tableofcontents
	
	\section{Introduction}
	This work was motivated by Gabber's purity theorem for pure lisse $\bar{\Q}_l$-sheaves \cite{BBD} and its complex counterpart for variations of Hodge structure \cite{Zuc,CKS,KK}. We searched for such an analogue for crystals in positive characteristic. In this paper, we have obtained some new results in this direction.
	
	In their foundational work in nonabelian Hodge theory in positive characteristic, Ogus-Vologodsky \cite[Corollary 2.27]{OV} established the following theorem:
	\begin{theorem}[Ogus-Vologodsky]\label{equivalence}
		Let $k$ be a perfect field of positive characteristic $p$.  Let $X$ be a smooth variety over $k$. Suppose $X$ is $W_2(k)$-liftable. Let $C$ be the Cartier transform from the category of $D^{\gamma}_X$-modules to the category of $\hat\Gamma_{\cdot}T_{X'}$-modules with respect to a chosen $W_2(k)$-lifting of $X$. Let $(H,\nabla)$ be a module with integrable connection over $X$ whose $p$-curvature is nilpotent of level $\ell\leq p-1$. Then there is an isomorphism in the derived category $D(X')$:
		$$
		F_{\ast}\tau_{<p-\ell}\Omega^{*}(H,\nabla)\cong \tau_{<p-\ell} \Omega^{*}(C(H,\nabla)),
		$$
		where $X'=X\times_{\sigma,k}k$, $\sigma$ is the Frobenius automorphism of $k$, $F: X\to X'$ is the relative Frobenius, $\Omega^{*}(H,\nabla)$ is the de Rham complex attached to $(H,\nabla)$, $\Omega^{*}(C(H,\nabla))$ is the Higgs complex attached to the Higgs module $C(H,\nabla)$.
	\end{theorem}
	It generalizes vastly the fundamental decomposition theorem of Deligne-Illusie \cite{DI} that is the statement for the de Rham object $(\sO_{X},d)$.  We shall call it the \emph{de Rham-Higgs comparison theorem for de Rham complexes in positive characteristic}. In a preprint \cite{Schepler}, Schepler extends the above result to the logarithmic setting. For the Gau{\ss}-Manin bundle associated to a semistable family over $k$, which is $W_2(k)$-liftable, the de Rham-Higgs comparison theorem of Ogus-Vologodsky was first established by Illusie \cite{IL90}.  Moreover, he proved the $E_1$-degeneration property and a relative Kodaira type vanishing theorem (that has been recently generalized to general coefficients \cite{Arapura}). In our work, we shall establish the ``intersection'' version of all these results.
	
	Our starting point is the observation that the intersection complex attached to a VHS $\V$ with unipotent local monodromies, first introduced in \cite[\S4.3]{KK} and expounded in \cite[Ch. 8]{CEGT}, can be rephrased in the Zariski topology if the underlying compact K\"{a}hler manifold is algebraic. Let $S$ be a regular locally Noetherian scheme $S$ and $X$ is a smooth $S$-scheme of finite relative dimension. Let $D$ be a reduced $S$-relative normal crossing divisor in $X$. For a flat logarithmic relative $\lambda$-connection $(E,\nabla)$ with $\lambda\in \Gamma(S,\sO_S)$ and simple pole of $\nabla$ along $D$, we introduce the notion of \emph{intersection $\lambda$-complex} associated to $(E,\nabla)$ which is denoted as $\Omega^*_{int}(E,\nabla)$. For $S=\Spec(\C)$ and $X$ smooth projective over $\C$, the complex analytification of the intersection $\lambda$-complex associated to the Deligne's canonical extension of a VHS with unipotent local monodromy $\V$ coincides with the intersection complex of Kashiwara-Kawai (see Lemma \ref{two def equ}). As a matter of convention, we call an intersection $\lambda$-complex for $\lambda=1$ (resp. $\lambda=0$) an \emph{intersection de Rham (resp. Higgs) complex}.
	
	Now we come back to the characteristic $p$ situation as we started. Our first main result is the following
	\begin{theoremA*}[Corollary \ref{dec thm}]\label{thm A}
		Let $X$ and $(H,\nabla)$ be as Theorem \ref{equivalence}. One has an isomorphism in $D(X')$:
		$$
		F_{\ast}\tau_{<p-\ell}\Omega_{int}^{*}(H,\nabla)\cong \tau_{<p-\ell} \Omega_{int}^{*}(C(H,\nabla)),
		$$
	\end{theoremA*}
	The above theorem in the case of $D=\emptyset$ is nothing but Theorem \ref{equivalence}. However, the method of its proof is different from that of Ogus-Vologodsky \cite{OV}. Rather in the spirit of Deligne-Illusie \cite{DI}, we construct a quasi-isomorphism between $\tau_{<p-\ell} \Omega^{*}(C(H,\nabla))$ and $\tau_{<p-\ell}\mathcal{C}(\mathcal{U}',F_{\ast}\Omega^{*}(H,\nabla))$ which also induces a quasi-isomorphism on intersection $\lambda$-complexes (which are subcomplexes of logarithmic $\lambda$-complexes). So our method gives another proof of Theorem \ref{equivalence} and its logarithmic generalization due to Schepler \cite{Schepler}.
	%\footnote{The proof of Ogus-Vologodsky \cite{OV} for Theorem \ref{equivalence} as well as that of Schepler \cite{Schepler} for its logarithmic analogue considered the torsor of local liftings of Frobenius. They did not need to resort to tedious \v{C}ech calculations as we did in this work. It is important to find a conceptual proof of Theorem A in future.}.
	
	Our next main result concerns the $E_1$-degeneration property of intersection de Rham complexes attached to \emph{periodic de Rham bundles}, whose Higgs counterpart was first introduced in \cite{LSZ1}. We remark that a logarithmic Gau{\ss}-Manin bundle attached to a semistable family, as considered in \cite{IL90}, is one-periodic (Proposition \ref{one-periodic}).  Analogous to the role played by polarization in the $E_1$-degeneration theorem due to Cattani-Kaplan-Schmid and Kashiwara-Kawai for VHS (see \cite[Theorem 8.3.25]{CEGT}), periodicity is essential to the following analogue in positive characteristic.
	\begin{theoremB*}[Corollary \ref{E_1 degeneration}]\label{thm B}
		Let $X$ be as  in Theorem \ref{equivalence} and assume additionally that $X$ is proper over $k$. For a periodic de Rham bundle $(H,\nabla,Fil)$ satisfying $\dim X+\rank(H)\leq p$, the spectral sequence associated to the filtered complex $\Omega_{int}^*(H,\nabla)$ degenerates at $E_1$.
	\end{theoremB*}
	A more precise form of the above theorem for one-periodic de Rham bundles appears in \S 4 (Theorem \ref{E_1 deg}). Besides the de Rham-Higgs comparison theorem, the other key ingredient to the $E_1$-degeneration theorem is the \emph{intersection adaptedness theorem} \footnote{The name is after the $L^2$ adaptedness theorem of S. Zucker \cite{Zuc}.} for periodic de Rham bundles (Theorem \ref{inter comm}).
	
	A good analogue in positive characteristic of the notion of variation of Hodge structures with unipotent local monodromies is the notion of strict $p$-torsion logarithmic Fontaine module, as introduced in \cite[Chapter IV, Section c)]{Fa} \footnote{In loc. cit., a smooth scheme $X_{W(k)}$ over $W(k)$ and a $W(k)$-relative simple normal crossing divisor $D_{W(k)}\subset X_{W(k)}$ are assumed. But the notion itself requires only a $W_2(k)$-lifting of the pair $(X',D')$, as first observed by Ogus-Vologodsky \cite[Definition 4.6]{OV} when $D=\emptyset$. See \S4.1 for detail.}. Theorem A and Theorem \ref{E_1 deg} combined together lead to the following structural result:
	\begin{corollary*}[Corollary \ref{FL module}]
		Let $X$ be as in Theorem B. Let $(H,\nabla, Fil, \Phi)$ be a strict $p$-torsion logarithmic Fontaine module over $S$ of Hodge-Tate weight $w$ with $w+\dim X\leq p-1$. Then for each $i$,  the natural inclusion of complexes $\Omega_{int}^*(H,\nabla)\to \Omega^*(H,\nabla)$ induces a morphism of strict $p$-torsion Fontaine-Laffaille modules \cite{FL}:
		$$
		\HH^i(X,\Omega_{int}^*(H,\nabla))\to \HH^i(X,\Omega^*(H,\nabla)).
		$$
	\end{corollary*}
	\begin{remark*}
		The corollary is in a good analogue to the classical result on the cohomology groups of VHS: let $\V$ be a variation of Hodge structure defined over $U=X-D$ where $X$ is a compact K\"{a}hler manifold and $D$ a reduced normal crossing divisor. Then for each $i$, the natural morphism of cohomology groups $IH^i(X,\V)\to H^i(U,\V)$ is a morphism of mixed Hodge structures.
	\end{remark*}
	Our last main result is a Kodaira-Saito vanishing theorem in positive characteristic.
	\begin{theoremC*}[Corollary \ref{vanishing result}]\label{thm C}
		Let $X$ be as in Theorem B and assume additionally that $X$ is projective over $k$. Let $(E,\theta)$ be a periodic Higgs bundle over $X$. If $\dim X+\rank(E)\leq p$, then the following holds
		$$
		\HH^i(X,\Omega_{int}^*(E,\theta)\otimes L)=0, \quad i>\dim X,
		$$
		for any ample line bundle $L$ on $X$.
	\end{theoremC*}
	The relation of the above theorem with \cite[Theorem 1 (ii)]{Arapura} is as follows: by \cite[Theorem 1.5]{LSZ}, a semistable Higgs bundle with vanishing Chern classes and with rank $\leq p$ is preperiodic. So our vanishing result gives the intersection analogue of Arapura's vanishing theorem for periodic Higgs bundles. We do not know how to extend the intersection adaptedness theorem to a general preperiodic de Rham bundle.
	
	Together with the results of Illusie \cite{IL90}, Theorems A, B, C shall form a useful package for the study of $W_2$-liftable semistable families in positive characteristic. Using the technique of spreading-out, we give an algebraic proof of the $E_1$-degeneration theorem of Cattani-Kaplan-Schmid \cite{CKS} and Kashiwara-Kawai \cite{KK}, and the vanishing theorem of Saito \cite{Saito} for a VHS of geometric origin. It is a meaningful problem to ask for the intersection analogue of $E_1$-degeneration theorem of Scholze (\cite[Theorem 1.6]{Scholze}, and its logarithmic analogue by Diao-Lan-Liu-Zhu  \cite[Theorem 1.7]{DLLZ}) in the rigid analytic setting.
	
	The structure of the paper is outlined as follows. In Section 2, we introduce the central notion of the article: intersection $\lambda$-complex. In Section 3, we establish the de Rham-Higgs comparison theorem for the intersection de Rham complex in positive characteristic. In Section 4, we establish the $E_1$-degeneration theorem for the intersection de Rham complex attached to a periodic de Rham bundle. In Section 5, we give some applications of the previous theory in characteristic zero. In Appendix, we construct an $\infty$-homotopy which is essential in our approach to the de Rham-Higgs comparison theorem.  
	
	{\bf Acknowledgement:} We would like to express our appreciation to Professor Arthur Ogus and Professor H\'{e}l\`{e}ne Esnault for valuable advices and comments on the paper.  We would like also to thank Professor Kang Zuo and Professor Luc Illusie for their interest and encouragement. Last but not least, we want to thank the members of Algebraic Geometry team at USTC and Professor Jilong Tong for numerous discussions.
	
	\section{Intersection $\lambda$-complexes}
	Let $S$ be a regular locally noetherian scheme and $\alpha: X\to S$ a smooth morphism of relative dimension $n$, and let $D\subset X$ be a reduced normal crossing divisor relative to $S$ (abbreviated as NCD in below).  Let $\lambda\in \Gamma(S, \sO_S)$, and $E$ be an $\sO_X$-module, equipped with a flat relative logarithmic $\lambda$-connection $\nabla$ with pole along $D$. Recall that a relative logarithmic $\lambda$-connection is an $S$-linear additive morphism
	$$
	\nabla: E\to E\otimes\Omega_{X/S}(\log D),
	$$
	satisfying the $\lambda$-Leibniz rule
	$$
	\nabla(fs)=s\otimes\lambda df+f\nabla(s),~f\in \sO_X,~s\in E.
	$$
	For simplicity, we denote $\Omega_{X/S}(\log D)$ by $\Omega_{X_{\log}/S}$ \footnote{There is a notion of sheaf of relative K\"{a}hler differentials in log geometry (\cite[\S1.7]{Kato}). In our case, we equip $X$ with the log structure determined by $D$ and $S$ with the trivial log structure (\S1.5 loc. cit.). Then we may extend the structural morphism $\alpha$ to a  morphism of log schemes $X_{\log}\to S$ which is smooth (\S3 loc. cit.).}, and speak of a $\lambda$-connection over $X_{\log}/S$, which means a relative logarithmic $\lambda$-connection with pole along $D$ as above. One forms the $\lambda$-complex $\Omega^*(E,\nabla)$ attached to $(E,\nabla)$ as follows:
	$$
	E\stackrel{\nabla}{\longrightarrow}E\otimes\Omega_{X_{\log}/S} \stackrel{\nabla}{\longrightarrow}E\otimes\Omega^2_{X_{\log}/S}\stackrel{\nabla}{\longrightarrow}\cdots,
	$$
	where $\Omega^i_{X_{\log}/S}=\bigwedge^i\Omega_{X_{\log}/S},i>0$ and
	$$
	\nabla(s\otimes\omega)=s\otimes\lambda d\omega+\nabla(s)\wedge\omega,~s\in E,~\omega\in \Omega^*_{X_{\log}/S}.
	$$
	Note that the above complex is naturally a graded module over the graded algebra $\Omega^{\bullet}_{X_{\log}/S}:=\bigoplus_{i=0}^{n}\Omega^i_{X_{\log}/S}$, whose module structure is determined by sending $\eta\in \Omega^i_{X_{\log}/S}$ and $e\otimes \omega\in E\otimes \Omega^j_{X_{\log}/S}$ to $e\otimes \omega\wedge \eta\in E\otimes \Omega^{i+j}_{X_{\log}/S}$. Since there is a natural morphism of graded algebras
	$$
	\Omega^{\bullet}_{X/S}:=\bigoplus_{i=0}^{n}\Omega^i_{X/S}\to \Omega^{\bullet}_{X_{\log}/S},
	$$
	we may regard $\Omega^{\bullet}(E,\nabla):=\bigoplus_{i=0}^nE\otimes \Omega^i_{X_{\log}/S}$ as a graded module over $\Omega^{\bullet}_{X/S}$. The main aim of this section is to construct the intersection $\lambda$-complex for $(E,\nabla)$, which is a subcomplex of $\Omega^*(E,\nabla)$ in the Zariski topology.  It will be achieved first in the simple normal crossing divisor (abbreviated as SNCD in below) case, and then in the NCD case using the technique of faithfully flat descent (\cite[VIII 1.]{SGA 1}).
	\subsection{The construction}
	Let $X$ a smooth projective variety over $\C$ and $D\subset X$ a NCD. Let $j: U\to X$ be the open immersion of the complement of $D$. For a polarizable variation of Hodge structure $\V$ (with unipotent local monodromies) over the complex manifold $U^{an}$, Kashiwara-Kawai \cite{KK86} constructed a resolution for the intermediate extension $j_{!*}\V_{\C}$, which is a certain subcomplex of the analytic logarithmic de Rham complex attached to the Deligne's canonical extension of $\V\otimes \sO_{U^{an}}$. It gives rise to an explicit construction of the Hodge filtration on the intersection cohomology group $H^i(X^{an},j_{!*}\V_{\C})$. This beautiful construction motivated us to consider its analogue in characteristic $p$. It turns out that such a construction does exist in a purely algebraic setting. As seen in later sections, it behaves well in the context of the nonabelian Hodge theory in positive characteristic. Before carrying out the general construction, let us explain the case of curves. Let $S, X, D, (E,\nabla)$ be as in the beginning of the section and $\dim_SX=1$. The intersection $\lambda$-complex for $(E,\nabla)$ reads
	$$
	E\stackrel{\nabla}{\longrightarrow}\nabla(E)+E\otimes\Omega_{X/S},
	$$
	which is a subcomplex of $E\stackrel{\nabla}{\longrightarrow}E\otimes \Omega_{X/S}(\log D)$, the $\lambda$-complex for $(E,\nabla)$.

	Recall that a Weil divisor $D\subset X$ is said to be simple normal crossing, if $D=\sum_iD_i$ with $D_i$ integral and regular, and different components intersect transversally. For a reduced NCD $D$, there is an \'{e}tale covering $X'\to X$ such that $D'=D\times_XX'$ is simple normal crossing. In our construction, it is essential to make a good choice of local coordinates. Let $D$ be a SNCD. Let $\A^n_S$ be endowed with the standard log structure given by the union $E_n$ of all coordinate hyperplanes. Then over a given point $x\in X$, one finds an open neighborhood $U\subset X$ together with an \'{e}tale and log \'{e}tale  morphism over $S$
	$$
	t=(t_1,\cdots,t_n): U\to \A^n_S,
	$$
	with $t^{-1}E_n=D$. We call $t$ an \emph{admissible system of local coordinates} on $U$.  For such a system of local coordinates on $U$,  the following conditions are satisfied:
	\begin{itemize}
		\item [(i)]  $\Omega_{U/S}$  admits an $\sO_U$-basis $\{dt_1,\cdots,dt_n\}$;
		\item [(ii)]  $\Omega_{U_{\log}/S}$  admits an $\sO_U$-basis  $\{\omega_1=d\log t_1,\cdots,\omega_n=d\log t_n\}$.
	\end{itemize}
	Clearly, for an \'{e}tale morphism $V\to U$, an admissible system of local coordinates $t$ on $U$ pulls back to an admissible system of local coordinates on $V$. For another admissible system of local coordinates $t'=(t'_1,\cdots,t'_n)$, shrinking $U$ if necessary to a smaller open neighborhood of $x$, we may find units $u_i$ such that, up to a permutation of indices,
	\begin{eqnarray}\label{system unique}
		t_i'=u_it_i.
	\end{eqnarray}
	We shall work with this type of local coordinates throughout the paper.

	Let $(E_U,\nabla_U)$ be the restriction of $(E,\nabla)$ to the open log subscheme $U_{\log}$. We may write the operator $\nabla_U$ into a sum of components
	$$\nabla_U=\sum_{1\leq i\leq n}\nabla_{U,i}\otimes\omega_i,$$
	with $\nabla_{U,i}\in \End_S(E_U)$. For  $\emptyset\neq I=\{i_1,\cdots,i_l\},i_1<\cdots<i_l$, set
	$$
	\nabla_{U,I}=\nabla_{U,i_1}\circ\cdots\circ \nabla_{U,i_l}, \  \omega_I=\omega_{i_1}\wedge\cdots\wedge\omega_{i_l},~d_It=dt_{i_1}\wedge\cdots\wedge dt_{i_l},~t_I=\prod_{i\in I}t_i.
	$$
	Set $\nabla_{U,\emptyset}=Id,\omega_{\emptyset}=d_\emptyset t=t_\emptyset=1\in \sO_U$.

	\begin{lemma}\label{independence of special local coord}
		Let $D$ be a SNCD and $U\subset X$ be an open subset such that admissible systems of local coordinates on $U$ exist.  Then the $\Omega^{\bullet}_{U/S}$-graded submodule ${R}_{t}$ of $\Omega^{\bullet}(E_U,\nabla_U)$, generated by the abelian subsheaf
		$$
		\sum_{I\subset \{1,\cdots,n\}}\nabla_{U,I}E_{U}\otimes \omega_I
		$$
		is independent of the choices of admissible systems of local coordinates on $U$. That is, for another admissible system of local coordinates $t'$ on $U$, ${R}_{t}={R}_{t'}$. Furthermore, the submodule $R_{t}$ is closed under the differential of the complex $\Omega^{*}(E_U,\nabla_U)$.
	\end{lemma}
	\begin{proof}
		 By symmetry, it suffices to show $R_{t'}\subset R_{t}$.  For every $x\in U$, we may find an open subset of $U$, such that the equality \ref{system unique} holds. Write
		$$
		\nabla_U=\sum_{1\leq i\leq n}\nabla'_{U,i}\otimes\omega_i',
		$$
		where $\omega_i'=d\log t'_i$. It suffices to show the following relation:
		\begin{eqnarray}\label{nabla' belong}
			\nabla'_{U,I}(e)\otimes\omega_I'\in R_{t},~I\subset \{1,\cdots,n\},~e\in E_U.
		\end{eqnarray}
		To this purpose, we need to rewrite $\nabla'_{U,I}$ (resp. $\omega_I'$) in terms of $\nabla_{U,J}$s (resp. $\omega_J$s) with $J\subset I$. To simply the notation, we drop the subscript $U$ in $\nabla_{U,i}$ etc.. Since $\omega_{i}=\omega'_{i}+u_{i}^{-1}du_{i}$, it follows that
		\begin{eqnarray}\label{relation nabla nabla'}\nabla'_i=\nabla_i+t_i\vartheta_i,~\vartheta_i=\sum_{1\leq j\leq n}b_j\nabla_j,~b_{j}\in \mathcal{O}_{U}.\end{eqnarray}
		It is easy to see that
		$$
		\omega_I'=\sum_{J\subset I}\omega_{J}\wedge\beta_{J},~\beta_{J}\in\Omega^{|I-J|}_{U/S}.
		$$
		To get an expression for $\nabla'_{I}$, we introduce some notations. Set for $1\leq i\leq n$, $\nabla^{0}_{i}=Id$ and for an $\alpha=(\alpha_1,\cdots,\alpha_n)\in \N^n$,
		$$
		\nabla^\alpha=\prod_{1\leq i\leq n}\nabla^{\alpha_i}_{i}.
		$$
		Further for any $m\geq0$, we set
		$$
		\Theta^m=\sum_{\alpha\in\mathbb{N}^n,|\alpha|\leq m}\Gamma(U,\mathcal{O}_U)\nabla^\alpha\subset\mathrm{End}_S(E_U)
		,$$
		and $\Theta=\bigcup_{m\geq0}\Theta^m$, the $\sO_U$-subalgebra generated by $\nabla_i$s. For any $I\subset\{1,\cdots,n\}$, we set
		$$
		\Theta_I=\sum_{J\subset I}t_{I-J}\nabla_{J}\Theta.
		$$
		We claim that $\Theta$ admits the following properties:
		\begin{itemize}
			\item[(i)\ ] $\Theta$ is a subring;
			\item[(ii)\ ] $\Theta t_i=t_i\Theta$;
			\item[(iii)\ ]$\Theta\nabla_{i}\subset \nabla_{i}\Theta+t_i\Theta$;
			\item[(iv)\ ] $\Theta_I\Theta_J=\Theta_{I+J}$ for $I,J\subset\{1,\cdots,n\}$ subject to the condition $I\cap J=\emptyset$.
		\end{itemize}
		Using \eqref{relation nabla nabla'} and (iv), we obtain
		\begin{eqnarray}\label{nabla' I}
			\nabla'_{I}=\sum_{J\subset I}t_{I-J}\nabla_{J}\vartheta_J,~\vartheta_J\in\Theta.
		\end{eqnarray}
		Now the requested \eqref{nabla' belong} follows from the following explicit expression:
		$$
		\begin{array}{c}\nabla'_{I}(e)\otimes\omega'_{I}=\sum_{J\subset I}t_{I-J}\nabla_{J}\vartheta_{J}(e)\otimes\sum_{J\subset I}\omega_{J}\wedge\beta_{J}=\\
			\sum_{J_{1},J_{2}\subset I}\nabla_{J_{1}\cap J_{2}}(\nabla_{J_1-J_{2}}\vartheta_{J_{1}}e)\otimes\omega_{J_{1}\cap J_{2}}\wedge t_{I-J_{1}-J_{2}}d_{J_{2}-J_{1}}t\wedge\beta_{J_{2}},
		\end{array}
		$$
		which is clearly an element in $R_{t}$. It remains to verify (i)-(iv): note for any $f\in \mathcal{O}_U$, it holds that
		\begin{eqnarray}\label{observation nabla}
			\nabla_{i}f=f\nabla_{i}+t_i\frac{\partial f}{\partial t_i}.
		\end{eqnarray}
		By induction on $m$, it follows that for $m>0$
		\begin{eqnarray}\label{Theta power}
			\Theta^m=(\Theta^1)^m.
		\end{eqnarray}
		This proves (i). Moreover, using \eqref{observation nabla}, we get $$\Theta^1t_i=t_i\Theta^1,$$
		which implies (ii) in conjunction with \eqref{Theta power}. (iii) is checked in a similar way. For (iv), it suffices to take $J=\{j\},j\notin I$ and then use (i)-(iii).
		
		Finally, $R_{t}$ is closed under the differential, because the action of the $\lambda$-connection $\nabla$ on an element $\nabla_{I}(e)\otimes \omega_I$ lies again in $R_{t}$, as shown in the formula:
		$$
		\nabla(\nabla_{I}(e)\otimes \omega_I)=\sum_{i\notin I}\pm\nabla_{I\cup\{i\}}(e)\otimes \omega_{I\cup\{i\}}.
		$$
		The lemma is proved.
	\end{proof}
	Let us make a further investigation into the structure of $R_{t}$. For any $w\in\mathbb{N}^{n}$, we define an $S$-submodule $E_w\subset E_U$ as follows:
	$$
	E_{w}=\sum_{\alpha+\beta=w,\alpha,\beta\in\mathbb{N}^{n}}t^{\alpha}\nabla^{\beta}E_U,
	$$
	where $t^{\alpha}=\prod_{i=1}^{n}t^{\alpha_{i}}_{i},\nabla^{\beta}=\prod_{i=1}^{n}\nabla^{\beta_{i}}_{i}$ for $\alpha=(\alpha_{1},\cdots,\alpha_{n}),\beta=(\beta_{1},\cdots,\beta_{n})$. We call a multi-index $w$ \emph{simple} if it is of the following form for some $I\subset\{1,\cdots,n\}$:
	$$w^I=(\delta^I_1,\cdots,\delta^I_n)\in\mathbb{N}^n,~\delta^I_i=|\{i\}\cap I|.$$
	For a simple multi-index $w^I$, $t^{w^I}=t_I$ and $\nabla^{w^I}=\nabla_I$. The following property is false for a general multi-index.
	\begin{lemma}\label{multi-weight E}
		For any simple $w\in\mathbb{N}^n$, $E_w$ is an $\mathcal{O}_U$-submodule of $E_U$.
	\end{lemma}
	\begin{proof}
		It suffices to show $\nabla^\beta E$ for $\beta$ simple admits the following property:
		\begin{eqnarray}\label{weight submodule beta}
			f\nabla^{\beta}E\in E_\beta,~f \in\mathcal{O}_{U}.
		\end{eqnarray}
		We prove it by induction on $|\beta|=\sum_{i=1}^r\beta_i$. When $|\beta|=0$, \eqref{weight submodule beta} holds trivially. Assume \eqref{weight submodule beta} holds for $|\beta|\leq N-1$. Given any $\beta=w^I$ with $|I|=N$, one picks an element $i\in I$ and writes $$\nabla^{\beta}=\nabla_{i}\nabla^{\beta'},\beta'=w^{I-\{i\}}.$$
		Since
		$$f\nabla_{i}e=\nabla_{i}(f e)-\lambda\frac{\partial f}{\partial t_{i}}t_{i}e,~f\in\mathcal{O}_{U},~e\in E,$$
		it follows that
		$$f\nabla^{\beta}e=\nabla_{i}(f\nabla^{\beta'}e)-\lambda\frac{\partial f}{\partial t_{i}}t_{i}\nabla^{\beta'}e.$$
		By the induction hypothesis, we get that
		$$f\nabla^{\beta'}e\in E_{\beta'},~\frac{\partial f}{\partial t_{i}}\nabla^{\beta'}e\in E_{\beta'}.$$
		As $\nabla_iE_{\beta'}\subset E_\beta$ and $t_iE_{\beta'}\subset E_\beta$,  \eqref{weight submodule beta} holds for $\beta$.  The lemma is proved.
	\end{proof}
	
	\begin{lemma}\label{coherent}
		For any simple $w\in\mathbb{N}^n$, $E_{w}$ is quasi-coherent (resp. coherent) if $E_U$ is quasi-coherent (resp. coherent).
	\end{lemma}
	\begin{proof}
		We may assume $U$ to be affine. For a quasi-coherent $\sO_U$-module $E_U$, we write $$E_U=\varinjlim_{E'\subset E_U~\mathrm{coherent}}E'.$$ Then the $\sO_U$-submodule $E'_{w}\subset E'$ is coherent. Indeed, for a set of sections $\{e_{1},\cdots,e_{m}\}\subset \Gamma(U,E)$ which generates $E'$, $E'_{w}$ is seen to be generated by the following sections:
		$$t^{\alpha}\nabla^{\beta}e_{i},~\alpha,\beta\in\mathbb{N}^{n},~\alpha+\beta=w,~1\leq i\leq m.$$
		This is because of the following formula which one proves by induction on $|\beta|$:
		\begin{eqnarray}\label{bring f outside}
			\nabla^{\beta}(fe)=\sum_{\alpha'+\beta'=\beta,\alpha',\beta'\in\mathbb{N}^{n}}f_{\beta'}t^{\alpha'}\nabla^{\beta'}e,
		\end{eqnarray}
		where $\beta$ is simple, $f, f_{\beta'}\in \sO_U$ and $e\in E_U$. As obviously $$E_w=\varinjlim_{E'\subset E~\mathrm{coherent}}E'_w,$$ it follows that $E_w$ is quasi-coherent. Certainly, if $E_U$ is coherent, the previous argument shows that $E_w$ is coherent.
	\end{proof}
	Note that the graded $\mathcal{O}_{U}$-module $\Omega^{\bullet}_{U_{\log}/S}$ has the following basis:
	\begin{eqnarray}\label{P}P^{\bullet}=\cup_{0\leq m\leq n}P^{m},~P^m=\{\omega_I: I\subset\{1,\cdots,n\},|I|=m\}.
	\end{eqnarray}
	\begin{definition}\label{mutli-weight}
		One defines the function $w:P^{\bullet}\to\{0,1\}^n$ as follows:
		$$w(\beta)=\left\{\begin{matrix}(0,\cdots,0)&\beta\in P^{0};\\
			(\epsilon_{1},\cdots,\epsilon_{n}),&\epsilon_{i}=|\{i\}\cap I|,~\beta=\omega_I\in P^{m},~m>0.\end{matrix}\right.$$
		We call $w(\beta)$ the multi-weight of $\beta$.
	\end{definition}
	We deduce immediately a key property of $R_{t}$ to the application of \'etale descent in the construction below.
	\begin{proposition}\label{multiweight Higgs}
		Regard $R_{t}$ in Lemma \ref{independence of special local coord} as an $\sO_U$-module. Then it decomposes into a direct sum of $\sO_U$-submodules:
		$$
		R_{t}=\bigoplus_{\beta\in P^{\bullet}}E_{w(\beta)}\otimes\beta.
		$$ Consequently, $R_{t}$ is a quasi-coherent (resp. coherent) $\sO_U$-module if $E_U$ is quasi-coherent (resp. coherent).
	\end{proposition}
	\begin{proof}
		It suffices to show the equality. For $I\cap J=\emptyset,f\in\mathcal{O}_{U}, e\in E_U$,
		$$f\nabla_{I}e\otimes \omega_{I}\wedge d_Jt=ft_{J}\nabla_{I}e\otimes\omega_{I\cup J}\in E_{w(\beta)}\otimes\beta,$$
		for $\beta=\omega_{I\cup J}$. The equality follows.
	\end{proof}
	Let $g:X'\to X$ be an \'etale morphism. Then (S)NCD $D\subset X$ pulls back to (S)NCD $D'\subset X'$. Let $X'_{\log}$ be the log scheme whose logarithmic structure is determined by $D'$. Let $(E',\nabla')$ be the pullback of $(E,\nabla)$ to $X'$ via $g$, which is an integrable $\lambda$-connection on $X'_{\log}/S$.
	\begin{lemma}\label{etale pull back}
		Let $D, U$ be as in Lemma \ref{independence of special local coord} and $g: U'\to U$ an \'{e}tale morphism. Let $t$ be an admissible system of local coordinates on $U$. Set $(E',\nabla')=g^*(E_U,\nabla_U)$ and $t'=g^*t$ the induced admissible system of local coordinates on $U'$.  Let $R_{t'}$ be the submodule of $\Omega^{\bullet}(E',\nabla')$ as defined by Lemma \ref{independence of special local coord}. Then one has
		$$
		R_{t'}=g^*R_{t}.
		$$
	\end{lemma}
	\begin{proof}
		As $g$ is flat, one gets an inclusion
		$$
		g^*R_{t}\subset g^*\Omega^{\bullet}(E_U,\nabla_U)=\Omega^{\bullet}(E',\nabla').
		$$
		Write $t=(t_1,\cdots,t_n)$ and $t'_i=g^*t_i\in \Gamma(U',\sO_{U'})$. The inclusion $g^*R_{t}\subset R_{t'}$ is obvious, because
		$$
		g^*(\nabla_{I}(e)\otimes\omega_I)=\nabla'_{I}(g^*e)\otimes\omega_I'\in R_{t'},~e\in E_U.
		$$
		By definition, any section of $R_{t'}$ is locally a finite $\sO_{U'}$-linear combination of
		$$
		\nabla'_{I} e'\otimes\omega'_{I}\wedge\beta',~e'\in E',~\beta'\in\Omega^{\bullet}_{U'/S}.$$
		We may write $e',\beta'$ in the above expression as a finite linear combination:
		$$e'=\sum_{j\in J_{1}}a_{j}g^*e_{j},~\beta'=\sum_{j\in J_{2}}b_{j}g^*\beta_{j},$$
		where $a_{j},b_{j}\in\mathcal{O}_{U'},~e_{j}\in E,\beta_{j}\in\Omega^{\bullet}_{U/S}$.
		By \eqref{bring f outside}, we have
		$$\nabla'_{I} e'=\sum_{J\subset I}c'_{J}t'_{I-J}\nabla'_{J}(g^*e_{J}),~c'_{J}\in\mathcal{O}_{U'},~e_{J}\in E_U,$$
		hence any section of $R_{t'}$ can be expressed as a finite linear combination of
		$$
		g^*(\nabla_Je\otimes\omega_J\wedge\beta),~e\in E_U,~\beta\in\Omega^{\bullet}_{U/S}.
		$$
		This shows the opposite inclusion $R_{t'}\subset g^*R_{t}$.
	\end{proof}
	
	For a given integrable $\lambda$-connection $(E,\nabla)$ over $X_{\log}/S$, the construction of the intersection $\lambda$-connection associated to $(E,\nabla)$ proceeds as follows:
	
	{\itshape Step 1: the SNCD case.} We assume $D$ to be a SNCD. Take a Zariski open covering $\sU=\{U_{\alpha}\}_{\alpha\in \Lambda}$, such that on each $U_{\alpha}$ there exists an admissible system of local coordinates $t_{\alpha}$.  Form the submodule $R_{t_{\alpha}}\subset \Omega^{\bullet}(E_{U_{\alpha}},\nabla_{U_{\alpha}})$ defined in Lemma \ref{independence of special local coord}. Then by Lemma \ref{independence of special local coord}, the various local modules $\{R_{t_{\alpha}}\}_{\alpha\in \Lambda}$ glue into a subcomplex of $\Omega^*(E,\nabla)$. It is independent of the choice of such an open covering $\sU$,  because it is stable under the refinement of open coverings.  We denote the subcomplex by $\Omega^*_{int}(E,\nabla)$.

	{\itshape Step 2: the quasi-coherent case.} Let $D$ be a reduced NCD and assume $E$ quasi-coherent. Take an \'{e}tale covering $g: X'\to X$ such that the pullback divisor $D'$ becomes a SNCD. Set $(E',\nabla')=g^*(E,\nabla)$. Let $p_i, i=1,2$ be the two natural projections $X'\times_X X'\to X'$ which are \'{e}tale morphisms. By Lemma \ref{etale pull back}, it follows that
	$$
	p_1^*\Omega^{*}_{int}(E',\nabla')=\Omega^{*}_{int}(E'',\nabla'')=p_2^*\Omega^{*}_{int}(E',\nabla')
	$$
	as subcomplexes of $\Omega^{*}(E'',\nabla'')$, where $(E'',\nabla'')$ is the pullback of $(E,\nabla)$ via the natural morphism $g\circ p_1=g\circ p_2: X'\times_X X'\to X$. By the \'etale descent for quasi-coherent modules, there is a unique submodule of $\Omega^{\bullet}(E,\nabla)$ which pulls back to $\Omega^{\bullet}_{int}(E',\nabla')$. The submodule is closed under differential because it is so after pullback. We denote the resulting subcomplex by $\Omega^*_{int}(E,\nabla)$.

	{\itshape Step 3: the general case.} Let $D$ be a reduced NCD and $(E,\nabla)$ an integrable $\lambda$-connection over $X_{\log}/S$. For $x\in X$, let $\sU_x=\{U_x\}$ be the directed set of Zariski open neighborhoods of $x$ with $\Omega_{U_{x,\log}/S}$ free. Then for any $U_x\in \sU_x$, we have a natural morphism of integrable $\lambda$-connections
	$$
	(\widetilde{E(U_x)},\widetilde \nabla)\to (E,\nabla)|_{U_x},
	$$
	where $\widetilde{E(U_x)}$ is the associated quasi-coherent $\sO_{U_x}$-module and $\widetilde \nabla$ is the natural $\lambda$-connection induced from $\nabla$. Let $\iota$ be the composite of the following natural morphisms
	$$
	\Omega^*_{int}(\widetilde{E(U_x)},\widetilde \nabla)_x\hookrightarrow \Omega^*(\widetilde{E(U_x)},\widetilde \nabla)_x\to \Omega^*(E,\nabla)_x.
	$$
	Then for each $x$, we define a sub germ of the complex by
	$$
	\sF^*_x:=\sum_{U_x\in \sU_x}\iota \Omega^*_{int}(\widetilde{E(U_x)},\widetilde \nabla)_x\subset \Omega^*(E,\nabla)_x.
	$$
	The set of germs $\{\sF^*_x\}_{x\in X}$ determines uniquely a subcomplex of $\Omega^*(E,\nabla)$ each component of which is a sheaf of $\sO_X$-submodule. We denote it by $\Omega^*_{int}(E,\nabla)$.
	
	It should be clear that the constructions in various steps are compatible. The so-constructed subcomplex $\Omega^*_{int}(E,\nabla)$ is called the intersection $\lambda$-complex associated to $(E,\nabla)$.
	
	\begin{remark}
		In a recent work \cite{Lin}, X.-J. Lin gives a coordinate-free reconstruction of $\Omega^*_{int}(E,\nabla)$ via the residue map of $\nabla$, under the condition that $D$ is a SNCD, $E$ is locally free of finite rank and the residues of $\nabla$ along various stratum, formed by intersections of components of $D$, are bundle morphisms.
	\end{remark}
	\subsection{Base change and cohomology}
	The following base change property for intersection $\lambda$-complexes follows by its construction.
	\begin{lemma}\label{etale base change}
		Let $g: X'\to X$ be an \'etale morphism. Let $(E,\nabla)$ be an integrable $\lambda$-connection over $X_{\log}/S$.Then there is an equality of subcomplexes
		$$
		\Omega^*_{int}g^*(E,\nabla)=g^*\Omega^*_{int}(E,\nabla).
		$$
	\end{lemma}
	\begin{proof}
		Set $(E',\nabla')=g^*(E,\nabla)$. If $D$ is a SNCD, then it is nothing but Lemma \ref{etale pull back}, because we may check the equality locally. If $D$ is a reduced NCD and $D'=D\times_XX'$ becomes a SNCD, then we may argue as follows:  let $\sF=\{(F,\nabla,\psi)\}$ be a family of morphisms of integrable $\lambda$-connections $\psi: (F,\nabla)\to (E,\nabla)$ with $F$ quasi-coherent. If $E_x=\sum_{F\in \sF}\iota(F_x)$, then by construction it holds that
		$$
		\Omega^{\bullet}_{int}(E,\nabla)_x=\sum_{F\in \sF}\iota \Omega^{\bullet}_{int}(F,\nabla)_x.
		$$
		Pick up any $x'\in X'$ and set $x=g(x')$. Clearly $E'_{x'}=E_x\otimes_{\sO_{x,X}}\sO_{x',X'}$.
		Therefore by Lemma \ref{etale pull back}, we have
		$$
		g^*\sum_{U_x\in \sU_x}\iota \Omega^{\bullet}_{int}(\widetilde{E(U_x)},\widetilde \nabla)_x=\sum_{U_x\in \sU_{x}}\iota  \Omega^{\bullet}_{int}(g^*(\widetilde{E(U_x)},\widetilde \nabla))_{x'},
		$$
		and then the equality $$
		g^*\Omega^{\bullet}_{int}(E,\nabla)_{x}=\Omega^{\bullet}_{int}(E',\nabla')_{x'}.$$
		Now we take an \'etale covering $h: X''\to X$ such that $D$ pulls back to a SNCD. Consider the following Cartesian diagram:
		$$
		\xymatrix{X'''\ar[r]^-{h'}\ar[d]_-{g'}&X'\ar[d]^-{g}\\
			X''\ar[r]^-{h}&X.}
		$$
		The previous discussions show that
		$$
		\begin{array}{lll}
			&&h'^*g^*\Omega^{\bullet}_{int}(E,\nabla)=g'^*h^*\Omega^{\bullet}_{int}(E,\nabla)=g'^*\Omega^{\bullet}_{int}h^*(E,\nabla) \\
			&=&\Omega^{\bullet}_{int}g'^*h^*(E,\nabla)=\Omega^{\bullet}_{int}h'^*g^*(E,\nabla)=h'^*\Omega^{\bullet}_{int}g^*(E,\nabla).
		\end{array}
		$$
		By the faithfully flat descent, $g^*\Omega^{\bullet}_{int}(E,\nabla)=\Omega^{\bullet}_{int}g^*(E,\nabla)$ as requested.
	\end{proof}
	Consider the following Cartesian diagram:
	$$\xymatrix{X'\ar[r]^-{g'}\ar[d]_-{f'}&X\ar[d]^-{f}\\
		S'\ar[r]^-{g}&S.}$$
	Set $D'=D\times_XX'$, which a reduced NCD of $X'$, and $\lambda'=g^*\lambda\in \Gamma(S',\sO_{S'})$. We endow $X'$ with the logarithmic structure determined by $D'$. An integrable $\lambda$-connection on $X_{\log}/S$ pulls back to an integrable $\lambda'$-connection $X'_{\log}/S'$ via $g'$.
	\begin{lemma}\label{base change}
		Let $(E,\nabla)$ be an integrable $\lambda$-connection on $X_{\log}/S$ and $S'$ a regular noetherian scheme. Let $(E',\nabla')$ be the pullback of $(E,\nabla)$ via $g'$ in the above Cartesian diagram. Then there is a natural surjective morphism $g'^*\Omega^{\bullet}_{int}(E,\nabla)\to \Omega^\bullet_{int}(E',\nabla')$, which becomes an isomorphism when one of the following conditions is satisfied:
		\begin{itemize}
			\item[(i)] $g$ is flat;
			\item[(ii)] the quotient $\sO_X$-module $Q^{\bullet}:=\Omega^{\bullet}(E,\nabla)/\Omega^{\bullet}_{int}(E,\nabla)$ is flat over $S$.
		\end{itemize}
	\end{lemma}
	\begin{proof}
		We may assume $D$ simple normal crossing. The image of the composite of the following natural morphisms
		$$
		\alpha:g'\Omega^\bullet_{int}(E,\nabla)\to g'^*\Omega^\bullet(E,\nabla)\cong\Omega^\bullet(E',\nabla')
		$$
		is easily seen to be $\Omega^\bullet_{int}(E',\nabla')$. The reasoning is similar to that of Lemma \ref{etale pull back}. Under either condition (i) or (ii), $\alpha$ is injective. The lemma follows.
	\end{proof}
	\begin{proposition}\label{cohomology and base change}
		Notation as Lemma \ref{base change}. Assume additionally that $f$ is proper. Then, for any $i\geq 0$
		\begin{eqnarray}\label{higher direct image base change}g^*\R^if_*\Omega^{*}_{int}(E,\nabla)\cong \R^{i}f'_{*}\Omega^{*}_{int}(E',\nabla')\end{eqnarray}
		holds if one of the following conditions is satisfied:
		\begin{itemize}
			\item[(i)] $E$ is quasi-coherent and $g$ is flat;
			\item[(ii)] $E$ is coherent, $Q^\bullet$ is flat over $S$ and the $\sO_S$-module  $\R^if_*\Omega^{*}_{int}(E,\nabla)$ is locally free of finite rank for any $i$.
		\end{itemize}
	\end{proposition}
	\begin{proof}
		The problem is local both on $S$ and $S'$. So we may assume they are affine. Set $S=\Spec(A)$ and $S'=\Spec(A')$. By Corollary \ref{multiweight Higgs}, $\sF^{*}:=\Omega^*_{int}(E,\nabla)$ is a finite complex of (quasi-)coherent $\sO_{X}$-modules when $E$ is (quasi-)coherent. By Lemma \ref{base change}, $\sF^{'*}:=g'^*\Omega^*_{int}(E,\nabla)=\Omega^*_{int}(E',\nabla')$ is a finite complex of (quasi-)coherent $\sO_{X'}$-modules. Let $\sU$ be a finite open affine covering of $X$ which pullbacks to a finite open covering $\sU'$ of $X'$. Let $\sC^{*,*}$ be the \v{C}ech double complex associated to $\sF^*$ with respect to $\sU$, and let $\sC^*$ be the total \v{C}ech complex. Similarly, let $\sC^{'*}$ be the total \v{C}ech complex associated to $\sF^{'*}$ with respect to $\sU'$. By construction, $\sC^{'*}$ is the base change of $\sC^*$ via $g'$.
		Now that
		$$
		\R^if_*\Omega^{*}_{int}(E,\nabla)\cong H^i(f_*\sC^*);\quad \R^if'_*\Omega^{*}_{int}(E',\nabla')\cong H^i(f'_*\sC^{'*}),
		$$
		Case (i) follows as cohomology commutes with flat base change. If $Q^{\bullet}$ is flat over $S$, then $\Omega^{\bullet}_{int}(E,\nabla)$ is flat over $S$ (in particular $E$ is $S$-flat). It follows that the above finite complex $\sC^*$ has flat components. Since its cohomology groups are all locally free by assumption, it follows that for any morphism $A\to A'$, the natural morphisms
		$$
		H^i(f_*\sC^*)\otimes_AA'\to H^i(f_*\sC^*\otimes_AA')=H^i(f'_*\sC^{'*})
		$$
		are isomorphisms for all $i$s. This can be shown by a simple descending induction on $i$. Hence Case (ii) follows. Alternatively, one may replace the complex $\sC^*$ with the so-called Grothendieck complex $\sK^*$, which is a finite complex of \emph{finitely generated} flat $A$-modules (see e.g. \cite[Lemma 6.7]{IL02}, or \cite[Lemma 1, \S5]{M}), from which Case (ii) follows too.
	\end{proof}
	
	\begin{definition}\label{intersection cohomology group}
		Let $S$ be a regular locally Noetherian scheme. Let $\alpha: X\to S$ be a smooth morphism and $D\subset X$ an $S$-relative reduced NCD. Let $(E,\nabla)$ be an integrable $\lambda$-connection over $X_{\log}/S$. The $i$-th higher direct image $\R^i\alpha_{*}\Omega_{int}^*(E,\nabla)$ is naturally an $\sO_S$-module, which we call the $i$-th intersection cohomology group associated to $(E,\nabla)$.
	\end{definition}

	\section{De Rham-Higgs comparison theorem}
	In this section and the next, let $k$ be a perfect field of characteristic $p>0$, $X$ a smooth variety of dimension $n$ over $k$ and $D\subset X$ a reduced NCD. We assume the pair $(X,D)$ is $W_2(k)$-liftable. We shall choose and then fix such a lifting $(\tilde X,\tilde D)$ in the discussion. This is because the Ogus-Vologodsky correspondence (\cite{OV}, \cite{Schepler}) depends on such a choice. Set $\sX/\sS=((X,D)/k,(\tilde X',\tilde D')/W_2(k))$, where $(\tilde X',\tilde D')$ is the fiber product $(\tilde X,\tilde D)\times_{W_2(k),\sigma}W_2(k)$. The main result of the section is Theorem \ref{dec thm}. Our method is the extension of that of \cite{DI} to coefficients.
	
\subsection{Homotopy on complexes}
 The aim of this subsection is to introduce the notion of an \emph{$\sL$-indexed $\infty$-homotopy} between two complexes, which allows us to reinterpret the decomposition theorem \cite{DI} as an $F_{*}(\sL_{\sX/\sS})$-indexed $\infty$-homotopy. Here $\sL_{\sX/\sS}$ refers to the Zariski sheaf of log Frobenius liftings over $X$, which assigns to an open subset $U\subset X$ the set of log Frobenius liftings $\tilde U:=\tilde X|_{U}\to \tilde X'$.  The crystalline nature of $\sL_{\sX/\sS}$ plays an essential role in theory of Ogus-Vologodsky \cite{OV}.

Let $(Y,\mathcal O_Y)$ be a ringed space. For $\mathcal F^*,\mathcal G^*$, two complexes of sheaves of $\mathcal O_Y$-modules, we let $\mathcal Hom^*_{\mathcal O_Y}(\mathcal F^*,\mathcal G^*)$ be the associated Hom complex. Explicitly, for $r\in\mathbb Z$,
$$
\mathcal Hom^r_{\mathcal O_Y}(\mathcal F^*,\mathcal G^*):=\prod_{i\in\mathbb Z}\mathcal Hom_{\mathcal O_Y}(\mathcal F^i,\mathcal G^{i+r}),
$$
and
$$
d_{\mathcal Hom}:\mathcal Hom^r_{\mathcal O_Y}(\mathcal F^*,\mathcal G^*)\to\mathcal Hom^{r+1}_{\mathcal O_Y}(\mathcal F^*,\mathcal G^*)
$$
sends $\{f_i\in\mathcal Hom_{\mathcal O_Y}(\mathcal F^i,\mathcal G^{i+r})\}_{i\in\mathbb Z}$ to
$$\{d_{\mathcal G^*}(-1)^{i+r}f_i+(-1)^{i+r+1}f_{i+1}d_{\mathcal F^*}\in\mathcal Hom_{\mathcal O_Y}(\mathcal F^i,\mathcal G^{i+r+1})\}_{i\in\mathbb Z}.
$$
Let $\sL$ be a sheaf of sets over $Y$ whose stalks are all nonempty (hence the projection to $Y$ of the associated espace \'etal\'e to $\sL$ is surjective). Recall the construction of the simplicial complex $\Delta_*(\mathcal L)$ attached to $\mathcal L$ as follows: For $r\geq0$, $\Delta_r(\mathcal L)$ is the sheaf associated to the presheaf of abelian groups, which assigns to an open subset $U\subset Y$ the free abelian group generated by elements of $\Gamma(U,\mathcal L^{r+1})$. The boundary operator
$$\partial:\Delta_{r+1}(\mathcal L)\to\Delta_r(\mathcal L),~r\geq0$$
is given as usual:
$$
(l_0,\cdots,l_{r+1})\mapsto\sum_{q=0}^{r+1}(-1)^q (l_0,\cdots,\widehat{l_q},\cdots,l_{r+1}).
$$
Set $\Delta_r(\mathcal L)=0$ for $r<0$.
\begin{definition}\label{defn infty homotopy}
 Notation as above. An $\mathcal L$-indexed $\infty$-homotopy from $\mathcal F^*$ to $\mathcal G^*$ is a morphism of complexes of sheaves of abelian groups
 \begin{eqnarray}\label{Delta complex infty}
 	\mathrm{Ho}:\Delta_*(\mathcal L)\to  \mathcal Hom^{*}_{\mathcal O_Y}(\mathcal F^*,\mathcal G^*).
 \end{eqnarray}
 In other words, $\mathrm{Ho}$ is a family of morphisms
 	$$
 	\mathrm{Ho}^r:\mathcal L^{r+1}\to \mathcal Hom^{-r}_{\mathcal O_Y}(\mathcal F^*,\mathcal G^*),~r\geq0
 	$$
 	such that
 	$$
 	\delta\circ \mathrm{Ho}^r=d_{\mathcal Hom}\circ\mathrm{Ho}^{r+1},
 	$$
 	and the images of $\mathrm{Ho}^0$ are morphism of complexes. Here $\mathcal L^{r+1}$ is the direct product of $(r+1)$ copies of $\mathcal L$ and for $f:\mathcal L^{r+1}\to\mathcal Hom^{-r}_{\mathcal O_Y}(\mathcal F^*,\mathcal G^*)$,
 	$$
 	\delta f:\mathcal L^{r+2}\to\mathcal Hom^{-r}_{\mathcal O_Y}(\mathcal F^*,\mathcal G^*),~(l_0,\cdots,l_{r+1})\mapsto\sum_{q=0}^{r+1}(-1)^qf(\cdots,\widehat{l_q},\cdots).
 	$$
 \end{definition}
The major usage of an $\mathcal L$-indexed $\infty$-homotopy in this paper is the following result.
 \begin{proposition}\label{Cech construction}
Notation as above. Suppose that for any $y\in Y$, there is a section $l_y$ of $\mathcal L$ around $y$ such that $\mathrm{Ho}^0(l_y)$ is a  quasi-isomorphism. Then $\mathcal F^*$ is isomorphic to $\mathcal G^*$ in the derived category of $\mathcal O_Y$-modules. Moreover, under this situation, any image of $\mathrm{Ho}^0$ is a quasi-isomorphism.
\end{proposition}
 \begin{proof}
Construct an open covering $\mathcal U$ of $Y$ as follows: Set
$$
I=\{(U,l)|U\subset Y~\mathrm{open},l\in\Gamma(U,\mathcal L)\},
$$
and $\mathcal U=\{U_i\}_{i\in I}$ with $U_i=U$ for $i=(U,l)\in I$. Using the family $\{\mathrm{Ho}^r\}_{r\geq 0}$, we may naturally construct a morphism of complexes
\begin{eqnarray}\label{Cech morphism}
\varphi:\mathcal F^*\to\check{\mathcal C}(\mathcal U,\mathcal G^*),
\end{eqnarray}
where $\check{\mathcal C}(\mathcal U,\mathcal G^*)$ is the total complex associated to the $\mathrm{\check{C}}$ech double complex $\check{\mathcal C}^*(\mathcal U, \mathcal G^*)$. Indeed, for
$i_0,\cdots,i_r$ in $I$, we define
$$
\varphi(r,s)_{i_0,\cdots,i_r}:\mathcal F^{r+s}\to j_{*}\mathcal G^s|_{U_{i_0,\cdots,i_r}}
$$
to be the adjoint of
$$
\mathrm{Ho}^r(l_{i_0}|_{U_{i_0,\cdots,i_r}},\cdots,l_{i_r}|_{U_{i_0,\cdots,i_r}}):\mathcal F^{r+s}|_{U_{i_0,\cdots,i_r}}\to\mathcal G^{s}|_{U_{i_0,\cdots,i_r}}.
$$
Here for $i=(U,l)\in I$, $l_i=l$, and $j:U_{i_0,\cdots,i_r}:=U_{i_0}\cap\cdots\cap U_{i_r}\hookrightarrow Y$ is the natural inclusion. Hence, we obtain
$$
\varphi(r,s)=\bigoplus_{i_0,\cdots,i_r}\varphi(r,s)_{i_0,\cdots,i_r}:\mathcal F^{r+s}\to\check{\mathcal C}^r(\mathcal U, \mathcal G^s)=\bigoplus_{i_0,\cdots,i_r\in I}j_{*}\mathcal G^s|_{U_{i_0,\cdots,i_r}}.
$$
It is a tautology that the so-constructed $\varphi$ is a morphism of complexes.

Now assume the assumption of this proposition. We reproduce the argument of \cite[\S2 (d)]{DI} as follows. It is a local problem to show that $\varphi$ is a quasi-isomorphism. So we may assume that $\mathcal L$ admits a global section and that there is a global section $l_{*}\in\Gamma(Y,\mathcal L)$ such that $\mathrm{Ho}^0(l_*)$ is a quasi-isomorphism. Consider the following commutative diagram
 	$$
 	\xymatrix{
 		&\mathcal G^*&\\
 		\mathcal F^*\ar[r]^-{\varphi}\ar[ru]^-{\mathrm{Ho}^0(l_*)}&\check{\mathcal C}(\mathcal U, \mathcal G^*)\ar[u]&\mathcal G^*\ar[l]_{\epsilon}\ar[lu]_{id},
 	}
 	$$
 	where $\epsilon$ is the augmentation morphism and the vertical arrow is induced by the refinement of coverings: Take the open covering consisting of only one element $\mathcal{U}_*=\{Y\}$ indexed by $(Y,l_{*})$. Therefore, all arrows in the diagram, including $\varphi$, are quasi-isomorphisms. The first statement follows. Now since $\varphi$ is known to be a quasi-isomorphism, running the commutative diagram again for any other section of $\sL$ tells us that its image under $\mathrm{Ho^0}$ is also a quasi-isomorphism. This completes the whole proof.	
 \end{proof}

The decomposition theorem of Deligne-Illusie \cite{DI} may be reinterpreted as an instance of $\sL$-indexed $\infty$-homotopy.
 \begin{theorem}[{\cite[Th\'eor\`eme 2.1]{DI}}]\label{DI case}
 	Set $\mathcal L=F_*\mathcal L_{\mathcal X/\mathcal S}$. There is an $\mathcal L$-indexed $\infty$-homotopy \begin{eqnarray*}\label{Ho for trivial}
 		\mathrm{Ho}:\Delta_*(\mathcal L)\to\mathcal Hom^*_{\mathcal O_{X'}}(\bigoplus_{i=0}^{p-1}\Omega^i_{X'_{\log}/k}[-i],\tau_{<p}F_*\Omega^*_{X_{\log}/k})
 	\end{eqnarray*}
 	such that $\mathrm{Ho}^0$ sends any section of $\mathcal L$ to a quasi-isomorphism.
 \end{theorem}

\subsection{An infinity homotopy for nilpotent Higgs modules}
We briefly recall the construction of the inverse Cartier transform $C^{-1}$ of Ogus-Vologodsky and Schepler via the technique of exponential twisting \cite{LSZ} (see Appendix \cite{LSYZ} for the SNCD case)\footnote{The inverse Cartier transform via the exponential twisting differs from the inverse Cartier transform of Ogus-Vologodsky by an automorphism $\iota$ of $\HIG_{\leq p-1}(X'_{\log}/k)$, which sends $(E,\theta)$ to $(E,-\theta)$.}. Take a Zariski open covering $\mathcal{U}=\{U_{\alpha}\}_{\alpha\in \varLambda}$ of $X$. By base change, one obtains the induced open covering  $\mathcal{U}'=\{ U_{\alpha}'\}_{\alpha\in \varLambda}$ of $X'$. For each $\alpha$, let $(\tilde{U}_{\alpha},\tilde{D}_{\alpha})$ be the restriction of the pair $(\tilde X,\tilde D)$ to $U_{\alpha}$. Set $\Omega_{U_{\alpha,\log}/k}$ to be the restriction of $\Omega_{X_{\log}/k}$ to $U_{\alpha}$. Recall that a \emph{logarithmic Frobenius lifting} over $\tilde U_{\alpha}$ is a $W_2(k)$-morphism
 $$\tilde{F}_{\alpha}: \sO_{\tilde U_{\alpha}}\to \sO_{\tilde U'_{\alpha}}$$
	lifting the relative Frobenius morphism on $U_{\alpha}$ and satisfying
$$
\tilde F_{\alpha}^*\sO_{\tilde U'_{\alpha}}(-\tilde D'_{\alpha})=\sO_{\tilde U_{\alpha}}(-p\tilde D_{\alpha}).
$$
Such a lifting exists locally and two such liftings differ by an element in $F^*T_{X_{\log}'/k}$ over $U_{\alpha}$, where $T_{X_{\log}'/k}$ is the $\sO_{X'}$-dual of  $\Omega_{X'_{\log}/k}$. Then we obtain an obstruction class in $H^1(X,F^*T_{X_{\log}'/k})$, whose \v{C}ech representative is given by
	$\{h_{\alpha\beta}: F^*\Omega_{U'_{\alpha\beta,\log}/k}\to \sO_{U_{\alpha\beta}}\}$ with the relations
	\begin{itemize}
		\item [(i)] Over $U_{\alpha\beta}$, $dh_{\alpha\beta}=\zeta_{\beta}-\zeta_{\alpha}$ with
		$\zeta_{\alpha}:=\frac{\tilde{F}_{\alpha}}{[p]}: F^{*}\Omega_{U'_{\alpha,\log}/k}\to \Omega_{U_{\alpha,\log}/k}$.
		\item [(ii)] Over $U_{\alpha\beta\gamma}$, $h_{\alpha\beta}+h_{\beta\gamma}=h_{\alpha\gamma}$.
	\end{itemize}
	In (i), the notation $\frac{1}{[p]}$ refers to the restriction of the following canonical isomorphism to $U$:
 $$
 p\Omega_{\tilde X/W_2(k)}(\log \tilde D)\cong \Omega_{X/k}(\log D).
 $$
 For a nilpotent Higgs module $(E,\theta)$ of level $\ell\leq p-1$ over $X'_{\log}/k$, an object in the category $\HIG_{\leq \ell}(X'_{\log}/k)$, its inverse Cartier transform $(H,\nabla):=C^{-1}(E,\theta)$, an object in the category $\MIC_{\leq \ell}(X_{\log}/k)$, is obtained by gluing local modules with integrable connection
	$$
	(H_{\tilde{F}_{\alpha}},\nabla_{\tilde{F}_{\alpha}}):=(F^*E|_{U_\alpha}, \nabla_{can}+(id\otimes\zeta_{\alpha})(F^{*}\theta))
	$$
	via the gluing data $\{G_{\alpha\beta}:=\exp(1\otimes h_{\alpha\beta})F^{*}\theta=\sum_{i=0}^{p-1}\frac{((1\otimes h_{\alpha\beta})F^{*}\theta)^{i}}{i!}\}$. It was shown that $(H,\nabla)$ is independent of the choices of data $(\tilde{U}_{\alpha},\tilde{F}_{\alpha})$. The following lemma follows directly from the construction.
	\begin{lemma}\label{inverse cartier and etale base change}
		The inverse Cartier transform commutes with \'etale base change. Precisely, let $\tilde g: \tilde Y\to \tilde X$ be an \'etale morphism, $(E,\theta)\in \HIG_{\leq l}(X_{\log}/k)$. Set $g=\tilde g\times k: Y\to X$. Then $g^*(E,\theta)\in \HIG_{\leq l}(Y_{\log}/k)$, where $Y_{\log}$ is equipped with the log structure determined by $D\times_XY$, and
		$$
		C^{-1}(g^*(E,\theta))=g^*(C^{-1}(E,\theta)).
		$$	
	\end{lemma}
	\begin{proof}
		By \'etaleness, there is one unique log Frobenius lifting over a given log Frobenius lifting downstairs. Using these log Frobenius liftings as part of the construction data for the inverse Cartier transform upstairs, the equality becomes transparent. We refer the reader for a detailed proof in \cite[Theorem 5.3 ]{La}.
	\end{proof}
	
The following result generalizes Theorem \ref{DI case}.
	\begin{theorem}\label{infty homotopy}
Notation as above. Then there is an $\sL:=F_*\mathcal L_{\mathcal X/\mathcal S}$-indexed $\infty$-homotopy $\mathrm{Ho}$ from $\tau_{<p-\ell}\Omega^{*}(E,\theta)$ to $\tau_{<p-\ell}F_{*}\Omega^{*}(H,\nabla)$
				$$
		\mathrm{Ho}^r:\mathcal L^{r+1}\to\mathcal Hom^{-r}_{\mathcal O_{X'}}(\tau_{<p-\ell}\Omega^{*}(E,\theta),\tau_{<p-\ell}F_{*}\Omega^{*}(H,\nabla))
		$$
		such that the images of $\mathrm{Ho}^0$ are quasi-isomorphisms. Furthermore, $\mathrm{Ho}$ restricts to an $\infty$-homotopy from $\tau_{<p-\ell}\Omega_{int}^{*}(E,\theta)$ to $\tau_{<p-\ell}F_{*}\Omega_{int}^{*}(H,\nabla)$ such that the images of $\mathrm{Ho}^0$ are also quasi-isomorphisms.
	\end{theorem}
\begin{proof}	
The construction of $\mathrm{Ho}$ is the content of the Appendix, especially Theorem \ref{appendix construction Ho} and Corollary \ref{infty homotopy intersection}. That the images of $\mathrm{Ho}^0$ are quasi-isomorphisms is the content of the Cartier isomorphism. We derive it in our setting as Theorem \ref{Cartier iso}.
\end{proof}

\begin{corollary}\label{dec thm}
Notation in Theorem \ref{infty homotopy}. Then we have an isomorphism in $D(X')$:
$$
\tau_{<p-\ell}F_*\Omega_{int}^*(H,\nabla)\cong\tau_{<p-\ell}\Omega_{int}^*(E,\theta).
$$
\end{corollary}
\begin{proof}
Combine Theorem \ref{infty homotopy} with Proposition \ref{Cech construction}.
\end{proof}

	\subsection{Cartier isomorphism}
	A simple way to interpret the classical Cartier isomorphism (see e.g. \cite[Theorem 9.14]{EV}, \cite[Theorem 3.5]{IL02}) is that it computes the cohomology sheaves of the Frobenius push-forward of the de Rham complex. A beautiful generalization was found by A. Ogus for an arbitrary (logarithmic) module with integrable connection (\cite[Theorems 1.2.1, 3.1.1]{Ogus04}). As the classical one, the generalized Cartier isomorphism of Ogus is of purely positive characteristic. Its relation with the de Rham-Higgs comparison theorem of Ogus-Vologodsky (with liftability assumption) is explained in \cite[Remark 2.30]{OV}. The Cartier isomorphism we are going to present, which is related to that remark, is essentially local (assuming liftings of schemes and Frobenius) but also compatible with the intersection condition.
	\begin{theorem} \label{Cartier iso}
		Assume that over $\tilde X_{\log}$ there is a log Frobenius lifting $\tilde F$. Let $\mathrm{Ho}^0$ be the morphism constructed in Theorem \ref{appendix construction Ho}. Then
		\begin{eqnarray}\label{Csartier iso log}			\mathrm{Ho}^0(\tilde F): \tau_{<p-\ell}\Omega^*(E,\theta)\rightarrow\tau_{<p-\ell}F_*\Omega^*(H,\nabla)
		\end{eqnarray}
is an isomorphism and which  restricts to a quasi-isomorphism of intersection subcomplexes				\begin{eqnarray}\label{Cartier iso int}
			\mathrm{Ho}^0(\tilde F): \tau_{<p-\ell}\Omega^*_{int}(E,\theta)\rightarrow\tau_{<p-\ell}F_*\Omega^*_{int}(H,\nabla).
		\end{eqnarray}
	\end{theorem}
	
The problem is local and thus we can assume that $X$ admits an admissible system of local coordinates $t=(t_1,\cdots,t_n)$. By Proposition \ref{Cech construction}, one can even assume that $\tilde F$ is standard with respect to $t$, i.e. $\tilde F(t_i)=t_i^p$ for $1\leq i\leq n$.  From now on, we fix the natural identifications
	\begin{eqnarray}\label{identification F_*H}(H,\nabla)=(H_{\tilde F},\nabla_{\tilde F}), \quad F_*\Omega^\bullet(H,\nabla)=E\otimes_{\sO_{X'}} F_*\Omega^\bullet_{X_{\log}/k}.\end{eqnarray}
	Clearly, \eqref{Csartier iso log} is obtained by truncating
	$$
	\varphi:\Omega^*(E,\theta)\to F_*\Omega^*(H,\nabla),~e\mapsto e\otimes1,~e\otimes\beta_1\wedge\cdots\wedge\beta_i\mapsto e\otimes\zeta_{\tilde F}(\beta_1)\wedge\cdots\wedge\zeta_{\tilde F}(\beta_i)
	$$
	at place $<p-\ell$. One can check that $\varphi$ restricts to
	$$
	\varphi_{int}:\Omega_{int}^*(E,\theta)\to F_*\Omega_{int}^*(H,\nabla).	
	$$
		Theorem \ref{Cartier iso} is an immediate consequence of
	\begin{proposition}\label{Cartier iso standard}
		$\varphi,\varphi_{int}$ are quasi-isomorphisms.
	\end{proposition}

Before doing that, we introduce the \emph{scalarization} of a nilpotent Higgs module. Let $(E,\theta)$ be a nilpotent Higgs module on $X_{\log}'/k$. Let $T_{X'_{\log}/k}$ be the $\sO_{X'}$-dual of $\Omega^1_{X'_{\log}/k}$. The Higgs structure on the $\sO_{X'}$-module $E$ gives rise to a morphism of $\sO_{X'}$-algebras:
$$
\rho_\theta:S^\bullet T_{X'_{\log}/k}\to\mathcal End_{\mathcal O_{X'}}(E),
$$
so that $E$ is an $S^\bullet T_{X'_{\log}/k}$-module. Let $S^\bullet T_{X'_{\log}/k} \oplus E$ be the trivial infinitesimal extension of $S^\bullet T_{X'_{\log}/k}$ by $E$ (so we obtain an infinitesimal deformation of the vector bundle $\Omega^1_{X'_{\log}/k}$). Let $\ker(\rho^+_{\theta})$ be the kernel of $\rho_{\theta}|_{S^+ T_{X'_{\log}/k}}$. It is an ideal of $S^\bullet T_{X'_{\log}/k} \oplus E$. We put
$$
R(E,\theta):=S^\bullet T_{X'_{\log}/k} \oplus E/\ker(\rho^+_{\theta})\cong \mathcal O_{X'}\oplus\rho_\theta(S^+ T_{X'_{\log}/k})\oplus E,
$$
which inherits an $\sO_{X'}$-algebra structure from $S^\bullet T_{X'_{\log}/k} \oplus E$. More importantly, there is a tautological Higgs field over $R(E,\theta)$: Note that the Higgs field $\theta$ is a global section of
$$
\rho_\theta(S^+ T_{X'_{\log}/k}) \otimes_{\sO_{X'}}\Omega^1_{X'_{\log}/k} \subset\mathcal End_{\mathcal O_{X'}}(E)\otimes_{\sO_{X'}}\Omega^1_{X'_{\log}/k}.
$$
So we obtain a global section $\theta$ of $R(E,\theta)\otimes_{\sO_{X'}}\Omega^1_{X'_{\log}/k}$.
Thus, we obtain the following Higgs structure on $R(E,\theta)$:
$$
R(E,\theta)\to R(E,\theta)\otimes_{\sO_{X'}}\Omega^1_{X'_{\log}/k},~r\mapsto r\theta,~r\in R(E,\theta),
$$
where $R(E,\theta)\otimes_{\sO_{X'}}\Omega^1_{X'_{\log}/k}$ endowed with a natural $R(E,\theta)$-module structure.  By abuse of notation, we denote this Higgs field again by $\theta$. 
\begin{definition}
The Higgs module $(R(E,\theta),\theta)$ is said to be the scalarization of $(E,\theta)$.
\end{definition}
We put an $(R(E,\theta),F_*\Omega^\bullet_{X_{\log}/k})$-bimodule structure on $R(E,\theta)\otimes_{\sO_{X'}} F_*\Omega^\bullet_{X_{\log}/k}$ as follows:
$$
(r',\eta')\cdot (r\otimes\eta):=(rr')\otimes(\eta\wedge\eta'),~r,r'\in R(E,\theta),~\eta,\eta'\in F_*\Omega^\bullet_{X_{\log}/k}.$$
Since $(E,\theta)$ is a direct summand of $(R(E,\theta),\theta)$, we may replace $(E,\theta)$ by its scalarization. Thus we may assume that
	\begin{itemize}
	\item[(i)\ ] $E$ is an $\mathcal O_{X'}$-algebra, $\mathcal O_{X'}\subset E$ as an $\mathcal O_{X'}$-subalgebra and there is a projection of $\mathcal O_{X'}$-algebras $E\to \mathcal O_{X'}$,
		\item[(ii)\ ] the projection above induces a morphism of Higgs modules $(E,\theta)\to (\mathcal O_{X'},0)$,
	\item[(iii)\ ] any section lies in the kernel of the projection above is nilpotent,
\item[(iv)\ ] $\theta:E\to E\otimes_{\mathcal O_{X'}}\Omega^1_{X'_{\log}/k},~e\mapsto e\cdot\vartheta,~\vartheta\in \Gamma(X',E\otimes\Omega^1_{X'_{\log}/k})$. \end{itemize}
Clearly, the assumptions above are satisfied by $(\mathcal O_{X'},0)$.	
	
Let us turn to the proof that $\varphi$ is a quasi-isomorphism. Write $\theta,\nabla$ as a finite sum of their components in the usual way. We abbreviate $\zeta_{\tilde F}$ as $\zeta$. By using $\omega'_i=d\log t'_i$ (resp. $\omega_i=d\log t_i$) for $1\leq i\leq n$, we define sets $P'^m,P'^\bullet$ (resp. $P^m,P^\bullet$) as in \eqref{P}. Let $B^{\bullet}=\cup_{0\leq m\leq n}B^m$, where $B^m$ is the subset of $\Gamma(X',F_{*}\Omega^m_{X_{\log}/k})$ consisting of elements of form
	$$
	t_1^{i_1}\cdots t_n^{i_n}\cdot\omega_I,~0\leq i_1,\cdots,i_n\leq p-1,~I\subset\{1,\cdots,n\}.
	$$
	Set $N^m=B^m-P^m$ and $N^\bullet=\cup_{0\leq m\leq n}N^m$. Set $\mathcal P^\bullet_\theta,\mathcal N^\bullet_\theta$ to be $E$-submodules of $E\otimes_{\sO_{X'}} F_*\Omega^\bullet_{X_{\log}/k}$ generated by $P^\bullet,N^\bullet$, respectively. Clearly, we have decompositions of $E$-modules
$$
\mathcal P^\bullet_\theta=\bigoplus_{0\leq m\leq n}\mathcal P^m_\theta,~\mathcal N^\bullet_\theta=\bigoplus_{0\leq m\leq n}\mathcal N^m_\theta,~F_*\Omega^m(H,\nabla)=\mathcal P^m_\theta\bigoplus\mathcal N^m_\theta.$$	
Note that the differential $\nabla$ of $F_*\Omega^*(H,\nabla)$ preserves $\mathcal P^\bullet_\theta$ as well as $\mathcal N^\bullet_\theta$. Set $\mathcal P^*_\theta:=(\mathcal P^\bullet_\theta,\nabla|_{\mathcal P^\bullet_\theta})$ and $\mathcal N^*_\theta:=(\mathcal N^\bullet_\theta,\nabla|_{\mathcal N^\bullet_\theta})$ which are subcomplexes of $F_*\Omega^*(H,\nabla)$. Obviously, $\varphi$ is an isomorphism onto $\mathcal P_\theta^*$. The following formula is the key input of Proposition \ref{Cartier iso standard}:
\begin{eqnarray}\label{differential formula}\nabla(1\otimes\beta)=\sum_{j\in J}(i_j
				+\theta_j)\otimes(\omega_j\wedge\beta),\end{eqnarray}
		where $e\in E,\ \beta=t_1^{i_1}\cdots t_n^{i_n}\cdot\omega_I\in B^{\bullet}$  and
		$J=\{1,\cdots,n\}-I$. Therefore, by the assumption (iv) of $(E,\theta)$, we have a decomposition of complexes of $E$-modules
$$
F_*\Omega^*(H,\nabla)=\mathcal P^*_\theta\bigoplus\mathcal N^*_\theta.
$$

	\begin{lemma}\label{acyclic lemma}
		  $\mathcal N_\theta^*$ is exact.
	\end{lemma}
It is crucial to consider first the case $(E,\theta)=(\mathcal O_{X'},0)$. Write $\mathcal P^*_0, \mathcal Q^*_0$ for $\mathcal P^*_\theta, \mathcal Q_\theta^*$ in this case. Let $V^\bullet$ be the $\mathbb F_p$-vector space of $\Gamma(X',F_*\Omega^\bullet_{X_{\log}/k})$ generated by $N^\bullet$. Clearly, $d (V^\bullet)\subset V^\bullet$. Let $V^*$ be the corresponding subcomplex of $\Gamma(X',F_*\Omega^*_{X_{\log}/k})$ with differential $d_{V^*}$. One observes that we have an identification of complexes
\begin{eqnarray}\label{N 0 V}
\mathcal N^*_0=\mathcal O_{X'}\otimes_{\mathbb F_p}V^*
\end{eqnarray}	
and an identification of graded $E$-modules
\begin{eqnarray}\label{N theta V}
~\mathcal N^\bullet_\theta=E\otimes_{\mathbb F_p}V^\bullet.\end{eqnarray}
Under the previous identification, we obtain $e\otimes_{\mathcal O_{X'}}\beta=e\otimes_{\mathbb F_p}\beta$ for $e\in E$ and $\beta\in N^\bullet$. We shall simply write $e\otimes \beta$ in what follows.
	\begin{sublemma}\label{sublemma1}
	The complex $V^*$ is exact. In particular, 
	\begin{eqnarray}\label{V^bullet}
	V^\bullet\stackrel{d_{V^*}}{\longrightarrow}V^\bullet\stackrel{d_{V^*}}{\longrightarrow}V^\bullet
	\end{eqnarray}
	is exact at the middle term. Furthermore, there exist decompositions
	$$
	N^\bullet=N^\bullet_{i1}\coprod N^\bullet_{i2}, ~i=1,2,3$$
	such that \eqref{V^bullet} becomes
	\begin{eqnarray}\label{decomposable exact}
	V^\bullet_{11}\bigoplus V^\bullet_{12}\xrightarrow{ \left(\begin{matrix}f_{11}&f_{12}\\
	f_{21}&f_{22}\end{matrix}\right)}{}V^\bullet_{21}\bigoplus V^\bullet_{22}\xrightarrow{ \left(\begin{matrix}g_{11}&g_{12}\\
	g_{21}&g_{22}\end{matrix}\right)}{}V^\bullet_{31}\bigoplus V^\bullet_{32}	\end{eqnarray}	
	where $f_{22},g_{11}$ are isomorphisms. Here
$$
f_{ij}:V^\bullet_{1j}\to V^\bullet_{2i},~g_{ij}:V^\bullet_{2j}\to V^\bullet_{3i},~i,j=1,2.
$$

	\end{sublemma}
	\begin{proof}
		 By the Cartier isomorphism (see \cite[Th\'eor\`eme 1.2]{DI}), one knows that $\mathcal N^*_0$ is exact. Hence $V^*$ is exact. Clearly, the exactness of $V^*$ is equivalent to \eqref{V^bullet} is exact at the middle place. Set $r:=\mathrm{rank}(d_{V^*})$ and put
$$
d_{V^*}(\beta)=\sum_{\gamma\in N^\bullet}a_{\gamma\beta}\gamma,~\beta\in N^\bullet,~a_{\gamma\beta}\in\mathbb F_p.
$$	
Basic linear algebra shows that there are subsets 
\begin{eqnarray}\label{basis for N}
\{\beta_1,\cdots,\beta_r\}\subset N^\bullet,~\{\gamma_1,\cdots,\gamma_r\}\subset N^\bullet
\end{eqnarray}	 
such that $\mathrm{det}(a_{\gamma_i\beta_j})_{1\leq i,j\leq r}\neq0$. Set $V^\bullet_{21},V^\bullet_{22},V^\bullet_{31},V^\bullet_{32}$ to be subspaces of $V^\bullet$ generated by
$$
N^\bullet_{21}:=\{\beta_1,\cdots,\beta_r\},~N^\bullet_{22}:=N^\bullet-\{\beta_1,\cdots,\beta_r\},~N^\bullet_{31}:=\{\gamma_1,\cdots,\gamma_r\},~N^\bullet_{32}:=N^\bullet-\{\gamma_1,\cdots,\gamma_r\}$$
over $\mathbb F_p$, respectively. Clearly, $g_{11}$ is an isomorphism. 

Note that the composite of $\mathrm{Ker}(d_{V^*})\subset V^\bullet\twoheadrightarrow V^\bullet_{22}$ is an isomorphism. Since $\mathrm{Im}(d)=\mathrm{Ker}(d)$, it follows that the composite of
$$
V^\bullet\stackrel{d_{V^*}}{\longrightarrow}V^\bullet\twoheadrightarrow V^\bullet_{22}
$$
is surjective. By a similar argument as above, there is a decomposition
$$
N^\bullet=N^\bullet_{11}\coprod N^\bullet_{12}, 
$$
such that the composite of
$$
V_{12}^\bullet\subset V^\bullet\stackrel{d_{V^*}}{\longrightarrow}V^\bullet\twoheadrightarrow V^\bullet_{22}
$$
is an isomorphism. Here $V^\bullet_{12}$	is a subspace of $V^\bullet$ generated by $N^\bullet_{12}$ over $\mathbb F_p$. Set $V^\bullet_{11}$ to be the subspace of $V^\bullet$ generated by $N^\bullet_{11}$ over $\mathbb F_p$. Clearly, $f_{22}$ is an isomorphism. This completes the proof.
	\iffalse
		Note that Theorem \ref{Cartier iso} holds for the case $(\sO_{X'},0)$ since it is nothing but the classical Cartier isomorphism Under the morphism $\tilde \varphi$, $\F_pP_*\subset EP_*$, which has zero differential, is mapped isomorphically onto $\F_pQ_*$. Thus $\F_pQ_*$ with zero differential is subcomplex of $\F_pB_*$. Because $\F_pB_*=\F_pQ_*\oplus \F_pN_*$ and $H^*(\F_pB_*)=H^*(\F_pP_*)=H^*(\F_pQ_*)$ (which results from the Cartier isomorphism), it follows that $H^*(\F_pN_*)=0$ and hence $\F_pN_*$ is exact and hence $V^*$ is exact.
		\fi
	\end{proof}
	\iffalse
	The next sublemma gives an explicit expression of the $\sO_{X'}$-linear map $d_{\theta}: EB_0\to EB_1$, whose proof is a straightforward calculation.
	\begin{sublemma}\label{explicit expression of differential}
		Write $\theta=\sum_{i=1}^{n}\theta_{i}\otimes \omega_i$ where the coefficients $\{\theta_i\}_{1\leq i\leq n}\subset \End_{\mathcal{O}_{X^{'}}}(E)$ are mutually commutative and nilpotent. Then it holds that
		\begin{eqnarray*}d_{\theta}(e\otimes t_{i}^{j})=\left\{\begin{array}{l}(j\cdot id_E+\theta_{i})(e)\otimes t_{i}^{j-1}dt_{i},~1\leq i\leq r,~0\leq j\leq p-1;\\
				(j\cdot id_E+t_{i}\theta_{i})(e)\otimes t_{i}^{j-1}dt_{i},~r+1\leq i\leq n,~1\leq j\leq p-1;\\
				\theta_{i}(e)\otimes t_{i}^{p-1}dt_{i},~r+1\leq i\leq n,~j=0.
			\end{array}\right.
		\end{eqnarray*}
	\end{sublemma}
	\fi
	We proceed now to the proof of Lemma \ref{acyclic lemma}.
	\begin{proof}
	 Clearly, the exactness of $\mathcal N^*_\theta$ is equivalent to 
\begin{eqnarray}\label{N theta}
\mathcal N^\bullet_\theta\stackrel{\nabla|_{\mathcal N^\bullet_\theta}}{\longrightarrow}\mathcal N^\bullet_\theta\stackrel{\nabla|_{\mathcal N^\bullet_\theta}}{\longrightarrow}\mathcal N^\bullet_\theta
\end{eqnarray}	
is exact at the middle term. By using the identification \eqref{N theta V}, \eqref{N theta} can be rewritten as
$$
E\otimes V^\bullet_{11}\bigoplus E\otimes V^\bullet_{12}\xrightarrow{ \left(\begin{matrix}\tilde f_{11}&\tilde f_{12}\\
	\tilde f_{21}&\tilde f_{22}\end{matrix}\right)}{}E\otimes V^\bullet_{21}\bigoplus E\otimes V^\bullet_{22}\xrightarrow{ \left(\begin{matrix}\tilde g_{11}&\tilde g_{12}\\
	\tilde g_{21}&\tilde g_{22}\end{matrix}\right)}{}E\otimes V^\bullet_{31}\bigoplus E\otimes V^\bullet_{32}.$$
By the assumption (ii) of $(E,\theta)$, $E\to\mathcal O_{X'}$ induces a projection from the above complex to $\mathcal N^*_0$. Consequently, $\mathrm{det}(\tilde g_{11}),\mathrm{det}(\tilde f_{22})\in\Gamma(X',E)$ are mapped to $\mathrm{det}( g_{11}),\mathrm{det}(f_{22})\in\Gamma(X',\mathcal O_{X'})$ which are units. By the assumption (iii) of $(E,\theta)$, $\mathrm{det}(\tilde g_{11}),\mathrm{det}(\tilde f_{22})\in\Gamma(X',E)$ are units. In particular, $\tilde g_{11},\tilde f_{22}$ are isomorphisms, from which we conclude that \eqref{N theta} is exact at the middle term. This completes the proof.
	\end{proof}
	Next we are going to adapt the previous arguments to the intersection complexes.

	\begin{definition}
		Let $\beta=t_1^{i_1}\cdots t_n^{i_n}\cdot\omega_I\in B^{m},1\leq m\leq n$. One defines the multi-weight of $\beta$ by
		$$w(\beta):=(\epsilon_{1},\cdots,\epsilon_n)\in\{0,1\}^n,$$
		where $\epsilon_{j}=1$ if $j\in I$ and $i_j=0$, otherwise $\epsilon_{j}=0$. For any $w=(w_{1},\cdots,w_n)\in \mathbb{Z}^n$, we define
	$$
	E_w:=E_{0\vee w}=\sum_{\alpha+\beta=0\vee w,\alpha,\beta\in \N^n}t^{\alpha}\theta^{\beta}E,~0\vee w=(\mathrm{max}\{0,w_1\},\cdots,\mathrm{max}\{0,w_n\})\in\mathbb N^n.
	$$	
	\end{definition}
	\begin{remark}\label{weight order}
		 For two elements $w,w'\in \Z^n$, we have $E_{w}E_{w'}\subset E_{w+w'}$.\end{remark}
	\iffalse
	The first merit of the notion is an easy identification of an intersection complex. We may identify $\Omega^{*}_{int}(E,\theta)\subset \Omega^*(E,\theta)$ with the following complex $EP_{int}^{*}\subset EP^{*}$:
	$$\xymatrix{\oplus_{\beta\in P^{0}} E_{w(\beta)}\beta\ar[r]^-{\theta\wedge}&\oplus_{\beta\in P^{1}} E_{w(\beta)}\beta\ar[r]^-{\theta\wedge}&\cdots\ar[r]^-{\theta\wedge}&\oplus_{\beta\in P^{n}} E_{w(\beta)}\beta},$$
	where $E_{w(\beta)}\beta=\F_p\beta\otimes_{\F_p}E_{w(\beta)}$. Similarly, $F_{*}(\Omega^*_{int}(H,\nabla))\subset F_*(\Omega^*(H,\nabla))$ can be identified with the following complex $EB_{int}_{*}\subset EB^*$:
	$$
	\xymatrix{\oplus_{\beta\in B^{0}} E_{w(\beta)}\beta\ar[r]^-{d_{\theta}}&\oplus_{\beta\in B^{1}} E_{w(\beta)}\beta\ar[r]^-{d_{\theta}}&\cdots\ar[r]^-{d_{\theta}}&\oplus_{\beta\in B^{n}} E_{w(\beta)}\beta}.
	$$
	\fi
	The key ingedients to show that $\varphi_{int}$ is a quasi-isomorphism are the following:
	\begin{itemize}
\item[(1)\ ]	Write $\nabla(1\otimes \beta)=\sum_{\gamma\in B^\bullet}\vartheta_{\gamma\beta}\otimes\gamma,~\beta\in B^\bullet$. Then $\vartheta_{\gamma\beta}\in E_{w(\gamma)-w(\beta)}$;	
\item[(2)\ ] $\sum_{\beta\in N_{12}^\bullet}w(\beta)=\sum_{\beta\in N_{22}^\bullet}w(\beta),~\sum_{\beta\in N_{21}^\bullet}w(\beta)=\sum_{\beta\in N_{31}^\bullet}w(\beta)$;
\item[(3)\ ] The decomposition $F_*\Omega^\bullet(H,\nabla)=\bigoplus_{\beta\in B^\bullet}E\otimes \beta$ is compatible with the intersection condition.
	\end{itemize}
The first two ingredients follow from the following information about $\vartheta_{\gamma\beta}$ which can be obtained directly from the formula \eqref{differential formula}: (a) it belongs to $E_{w(\gamma)-w(\beta)}$, (b) it is either invertible or nilpotent, and (c) $w(\beta)=w(\gamma)$ if it is a unit. Clearly, (a) implies (1). By using \eqref{basis for N} and the proof of Lemma \ref{acyclic lemma}, we know that $\mathrm{det}(\vartheta_{\gamma_i\beta_j})_{1\leq i,j\leq r}=\mathrm{det}(\tilde g_{11})$ is a unit. According to (b), there is a permutation $\sigma$ of $\{1,\cdots,n\}$ such that $\prod^n_{i=1}\vartheta_{\gamma_i\beta_{\sigma(i)}}$ is a unit. Taking (c) into account, one has 
$$\sum_{\beta\in N_{12}^\bullet}w(\beta)=\sum_{i=1}^nw(\beta_{\sigma(i)})=\sum_{i=1}^nw(\gamma_i)=\sum_{\beta\in N_{22}^\bullet}w(\beta).$$
The other equality in (2) can be obtained similarly. So (2) follows. (3) is the following 
	\begin{lemma}\label{identification int}
				$F_*\Omega^\bullet_{int}(H,\nabla)=\bigoplus_{\beta\in B^\bullet}E_{w(\beta)}\otimes \beta$.
			\end{lemma}
	\begin{proof}
This lemma follows from the following facts:
\begin{itemize}
\item $F_*\Omega_{int}^\bullet(H,\nabla)$ is an $F_*\Omega_{X/k}^\bullet$-submodule of $F_*\Omega^\bullet(H,\nabla)$;
\item $\mathrm{RHS}:=\bigoplus_{\beta\in B^\bullet}E_{w(\beta)}\otimes\beta$  is an $F_*\Omega_{X/k}^\bullet$-submodule of $F_*\Omega^\bullet(H,\nabla)$;
\item As $F_*\Omega_{X/k}^\bullet$-module, $F_*\Omega_{int}^\bullet(H,\nabla)$ and $\bigoplus_{\beta\in B^\bullet}E_{w(\beta)}\otimes \beta$ are generated by $$\sum_{I\subset\{1,\cdots,n\}}\theta_IE\otimes\omega_I.$$
\end{itemize}				
Let us check the facts above. The first fact is obvious. Note that as an $\mathcal O_{X'}$-algebra, $F_*\Omega^\bullet_{X_{\log}/k}$ is generated by
$t_i,t_i\omega_i,1\leq i\leq n$. Thus the second fact follows from the following obvious inclusions:
$$
\mathrm{RHS}\cdot t_i\subset\mathrm{RHS},\quad \mathrm{RHS}\wedge(t_i\omega_i)\subset \mathrm{RHS},~1\leq i\leq n.$$		
For the last fact, it suffices to observe that $[t_i\partial_i,F^*\theta_j]=0$ for $i\neq j$ and
$$
H_w=\sum_{\alpha+\beta=w,\alpha,\beta\in\mathbb N^n}t^\alpha\prod^n_{i=1}(t_i\partial_i+F^*\theta_i)^{\beta_i} F^*E=\sum_{\alpha+\beta=w,\alpha,\beta\in\mathbb N^n}t^\alpha\prod_{i=1}^n(F^*\theta_i)^{\beta_i} F^*E,$$
where $\nabla_i=t_i\partial_i+F^*\theta_i$ for $1\leq i\leq n$ and $\beta=(\beta_1,\cdots,\beta_n)$.
		This completes the proof.
	\end{proof}	
The above lemma shows that $\varphi_{int}$ is an isomorphism onto $\mathcal P^*_\theta\cap F_*\Omega_{int}^*(H,\nabla)$ and
$$
F_*\Omega_{int}^*(H,\nabla)=(\mathcal P^*_\theta\cap F_*\Omega_{int}^*(H,\nabla))\bigoplus(\mathcal N^*_\theta\cap F_*\Omega_{int}^*(H,\nabla)).$$
Set $\mathcal N_{\theta,int}^*:=\mathcal N^*_\theta\cap F_*\Omega_{int}^*(H,\nabla)$. In order to show that $\varphi_{int}$ is a quasi-isomorphism, it remains to show the following 
	\begin{lemma}\label{acyclic int}
		$\mathcal N_{\theta,int}^{*}$ is exact.
	\end{lemma}
	\iffalse
	The second half of Theorem \ref{Cartier iso} follows from the next lemma.
	\begin{lemma}\label{lemma IC}
		One has the decomposition $EB_{int}^{*}=EQ_{int}^{*}\oplus EN_{int}^{*}$, where $\tilde \varphi_{int}$ induces isomorphism $EP_{int}^{*}\cong EQ_{int}^{*}$ and $EN_{int}^{*}$ is exact.
	\end{lemma}
	\fi
	\begin{proof}
Note that $\mathcal N_{\theta,int}^{*}=\bigoplus_{\beta\in N^\bullet}E_{w(\beta)}\otimes\beta$.	According to the proof of Lemma \ref{acyclic lemma}, it suffices to show that $\tilde g_{11}$ (resp. $\tilde f_{22}$) induces an isomorphism from $\bigoplus_{\beta\in N_{21}^\bullet}E_{w(\beta)}\otimes\beta$ to $\bigoplus_{\beta\in N_{31}^\bullet}E_{w(\beta)}\otimes\beta$ (resp. from $\bigoplus_{\beta\in N_{12}^\bullet}E_{w(\beta)}\otimes\beta$ to $\bigoplus_{\beta\in N_{22}^\bullet}E_{w(\beta)}\otimes\beta$). We prove this statement for $\tilde g_{11}$ (it can be discussed similarly for $\tilde f_{22}$). In other words, we need to show that 
$$
\sum_{i=1}^ru_i\otimes\beta_i:=\tilde g_{11}^{-1}(\sum_{i=1}^rv_i\otimes\gamma_i)\in \bigoplus_{i=1}^rE_{w(\beta_i)}\otimes\beta_i=\bigoplus_{\beta\in N_{12}^\bullet}E_{w(\beta)}\otimes\beta$$
for $v_i\in E_{w(\gamma_i)}$	with $1\leq i\leq r$. By symmetry, it suffices to check that $u_1\in E_{w(\beta_1)}$. Set $(\vartheta^\sharp_{ij}):=(\vartheta_{\gamma_i\beta_j})^{-1}$. For $r=1$, there is nothing to prove. Assume $r\geq2$. By basic linear algebra, we have
$$
u_1=\sum_{j=1}^r\vartheta^\sharp_{1j}v_j.
$$
We claim that $\vartheta^\sharp_{1j}v_j\in E_{w(\beta_1)}$ holds for $1\leq j\leq r$. By symmetry again, it suffices to show that $\vartheta^\sharp_{11}v_1\in E_{w(\beta_1)}$. Since
$$
\vartheta^\sharp_{11}v_1=\sum_{\sigma}\mathrm{sgn}(\sigma)\prod_{i=2}^r\vartheta_{\gamma_i\beta_{\sigma(i)}}v_1,
$$
coupled with the key ingredients (1), (2) and Remark \ref{weight order}, one has
$$
\vartheta^\sharp_{11}v_1\in\sum_{\sigma}\prod_{i=2}^rE_{w(\gamma_i)-w(\beta_{\sigma(i)})}E_{w(\gamma_1)}\subset\sum_{\sigma}E_{w(\beta_1)}=E_{w(\beta_1)}.
$$
Here $\sum_\sigma$ means the summation over all permutations of $\{2,\cdots,r\}$. This completes the proof.
	\end{proof}

	\section{$E_1$-degeneration theorem}
	Whether the Hodge to de Rham spectral sequence associated to a certain de Rham complex degenerates at $E_1$ is one of central problems in Hodge theory. We shall address it in this section, as a natural continuation of our study on de Rham-Higgs comparison theorem.
	
	\subsection{Periodic de Rham bundles}
	Let us first remark that the works \cite{LSZ}, \cite{LSZ1} extend in an obvious way to the current logarithmic setting\footnote{As we have demonstrated in Appendix \cite{LSYZ}, which extends the major part of \cite{LSZ} in the case of SNCD, the extension can be realized by replacing Frobenius liftings by log Frobenius liftings and the Taylor formula by its log analogue (\cite[Formula 6.7.1]{Kato}).}. In the next definition, which was introduced in \cite{LSZ1} when $D$ is absent, the notation $C^{-1}$ (by abuse of notation) is actually the composite of the inverse Cartier transform using the construction in \S3.1 with the base change functor $\pi^*$, where $\pi:X'\to X$ is the natural morphism.
	\begin{definition}\label{logarithmic periodic de Rham}
		A de Rham module over $X_{\log}/k$ is a triple $(H,\nabla,Fil)$, where $H$ is a coherent $\sO_X$-module, $\nabla$ is logarithmic integrable connection on $H$ and $Fil$ is a Griffiths transverse filtration on $(H,\nabla)$. It is said to be periodic, if there exist a positive integer $f$ and a sequence of Griffiths transverse filtrations $\{Fil_{i}\}_{0\leq i\leq f-1}$ of level $\leq p-1$ inductively defined on
		$$
		(H_{i},\nabla_{i})=C^{-1}\circ Gr_{Fil_{i-1}}(H_{i-1},\nabla_{i-1})
		$$
		with $(H_0,\nabla_0,Fil_0)=(H,\nabla,Fil)$ such that the module with integrable connection $(H_{f},\nabla_{f})$ is isomorphic to the initial module with integrable connection $(H,\nabla)$. A graded Higgs module $(E,\theta)\in \HIG_{\leq p-1}(X_{\log}/k)$ is said to be periodic if $E$ is coherent and there exists a positive number $f$ and a sequence of Griffiths transverse filtrations $\{Fil_{i}\}_{0\leq i\leq f-1}$ of level $\leq p-1$ inductively defined on
		$$
		(H_{i},\nabla_{i})=C^{-1}(E_{i},\theta_{i})
		$$
		with $(E_0,\theta_0)=(E,\theta)$ and $(E_{i+1},\theta_{i+1})=Gr_{Fil_i}(H_{i},\nabla_{i})$ such that $(E_{f},\theta_{f})$ is isomorphic to the initial Higgs module $(E,\theta)$ as graded Higgs modules. If the positive integer $f$ in the definition can be chosen to be one, we call it to be one-periodic.
	\end{definition}
	   By the equivalence of categories Theorem \ref{equivalence} (i), the graded Higgs module $Gr_{Fil}(H,\nabla)$ for a periodic de Rham module $(H,\nabla,Fil)$ is periodic and conversely, the de Rham module $(C^{-1}(E,\theta),Fil_0)$ appearing in the definition of a periodic Higgs module is periodic. By a similar argument as \cite[Proposition 3.12]{LSZ1}, a periodic Higgs module over $X_{\log}/k$ is locally free, and a periodic de Rham module is also locally free (and the filtration is a filtration of locally free submodules and locally split). Therefore, we shall speak of periodic de Rham/Higgs bundles instead of modules from now on.
	
	\begin{lemma}\label{periodicity stable under etale base change}
		Let $\tilde g: \tilde Y\to \tilde X$ be an \'etale morphism over $W_2(k)$ and $g$ its mod $p$ reduction. Let $(H,\nabla,Fil)$ (resp. $(E,\theta)$) be a (one-)periodic de Rham (resp. Higgs) module over $X_{\log}/k$. Then its pull-back via $g$ is again (one)-periodic.
	\end{lemma}
	\begin{proof}
		This is a direct consequence of Lemma \ref{inverse cartier and etale base change}.
	\end{proof}
	The following definition is the logarithmic analogue of \cite[Definition 4.6]{OV}.
	\begin{definition}
	A strict $p$-torsion logarithmic Fontaine module over $X_{\log}/k$ is a quadruple $(H,\nabla,Fil,\psi)$, where $(H,\nabla,Fil)$ is a de Rham bundle over $X_{\log}/k$ and
	$$
	\psi: C^{-1}\circ Gr_{Fil}(H,\nabla)\cong (H,\nabla)
	$$
	is a morphism of modules with integrable connection.
	\end{definition}
In the paragraph following \cite[Definition 4.6]{OV}, Ogus-Vologodsky pointed out the truth of the following result when $D=\emptyset$.  Let $X_{W(k)}$ be a smooth $W(k)$-scheme and $D_{W(k)}\subset X_{W(k)}$ be a reduced normal crossing divisor relative to $W(k)$. Faltings introduced the category $MF^{\nabla}_{[0,p-1]}(X_{W(k)},D_{W(k)})$ in \cite[Chapter IV, Section c)]{Fa} (see \cite[Section 1]{SXZ} for more details.).
\begin{proposition}\label{one-periodic}
	Let $(H,\nabla,Fil,\Phi)$ be a strict $p$-torsion object in $MF^{\nabla}_{[0,p-1]}(X_{W(k)},D_{W(k)})$. Then the relative Frobenius $\Phi$ induces a global isomorphism
	$$
	\psi: C^{-1}\circ Gr_{Fil}(H,\nabla)\cong (H,\nabla),
	$$
	so that $(H,\nabla,Fil,\psi)$ is a strict $p$-torsion logarithmic Fontaine module over $X_{\log}/k$, where $C^{-1}$ respects the $W_2(k)$-lifting $(X_{W(k)}\times W_2(k), D_{W(k)}\times W_2(k))$. In particular, $(H,\nabla,Fil)$ forms a one-periodic de Rham bundle over $X_{\log}/k$.
	\end{proposition}
	\begin{proof}
	The remark of Ogus-Vologodsky has been realized in \cite[Proposition 4.1]{LSZ}. The same proof works verbatim in the logarithmic case, sufficing to replacing the Taylor formula used therein with its logarithmic analogue \cite[Formula 6.7.1]{Kato}.
	\end{proof}

	\subsection{Intersection adaptedness theorem}
In this subsection, we are going to establish an intersection adaptedness theorem for one-periodic de Rham bundles, which is the intersection version in positive characteristic of the $L^2$-adaptedness theorem in the complex analytic setting which was shown in the one-variable case by Zucker (\cite[\S5]{Zuc}) (his tool is the $SL_2$-orbit theorem of Schmid \cite{Schmid}).

 Let $(H,\nabla,Fil)$ be a de Rham bundle over $X_{\log}/k$, where the level of $Fil$ is assume to be $l$.  We extend the filtration on $H$ by setting
 $$
 Fil^{i}=H, \quad i<0; \quad Fil^{l+i}=0, \quad i>0.
 $$
 There is a natural filtration
 $$
 \Omega^*(H,\nabla)=Fil^0\Omega^*(H,\nabla)\supset Fil^1\Omega^*(H,\nabla)\supset \cdots \supset Fil^{l+n}\Omega^*(H,\nabla)\supset 0
 $$
 of the complex $\Omega^*(H,\nabla)$, where
 $$
 Fil^i\Omega^j(H,\nabla):=Fil^{i-j}\otimes \Omega^{j}_{X_{\log}/k}.
 $$
 Since $\Omega_{X_{\log}/k}$ is locally free, there is a natural identification of complexes
 $$
 \Omega^*Gr_{Fil}(H,\nabla)=Gr_{Fil}\Omega^*(H,\nabla).
 $$
 Now we endow $\Omega^*_{int}(H,\nabla)\subset \Omega^*(H,\nabla)$ with the induced filtration.  Then, it is natural to ask whether the following commutativity holds:
	$$
	\Omega^*_{int}Gr_{Fil}(H,\nabla)=Gr_{Fil}\Omega^*_{int}(H,\nabla),
	$$
	where both sides are regarded as subcomplexes of $ \Omega^*Gr_{Fil}(H,\nabla)$.  The equality holds away from the divisor $D$ at infinity, but not around $D$ in general, as the following example shows.
	\begin{example}
		Take $X=\Spec(k[t])$ and $D=\textrm{div}(t)$. Set $H=\mathcal{O}_Xe_1\oplus\mathcal{O}_Xe_2$,
		$$
		Fil^0H=H,~ Fil^1H=\mathcal{O}_Xe_1,~ Fil^2H=0.
		$$
		Define a logarithmic connection $\nabla$ over $H$ by the following formula:
		$$
		\nabla(e_1,e_2)=(e_1,e_2)\left(\begin{matrix}0 & d\log t\\
			0 & 0
		\end{matrix}\right).
		$$
		One computes that the degree one term of $\Omega^*_{int}(H,\nabla)$ equals
		$$
		Fil_1H\otimes\Omega_{X/k}(\log D)+H\otimes \Omega_{X/k} (\subset H\otimes\Omega_{X/k}(\log D)).
		$$
		So the degree one term of $Gr_{Fil}\Omega^*_{int}(H,\nabla)$ takes the form
		$$
		E^{1,0}\otimes \Omega_{X/k}(\log D)+ E\otimes \Omega_{X/k},
		$$
		where $(E=E^{1,0}\oplus E^{0,1},\theta)$ is the associated graded Higgs module. But since $\theta=0$ as seen easily, it follows that the degree one term of $\Omega^*_{int}Gr_{Fil}(H,\nabla)$ takes the form
		$E\otimes \Omega_{X/k}$, which is a strict submodule of the degree one term of $Gr_{Fil}\Omega^*_{int}(H,\nabla)$.

	\end{example}

	\begin{theorem}\label{inter comm}
		Let $(H,\nabla,Fil)$ be a one-periodic de Rham bundle. Then the following equality of subcomplexes of $ \Omega^*Gr_{Fil}(H,\nabla)$ holds:
		$$Gr_{Fil}\Omega_{int}^{*}(H,\nabla)=\Omega_{int}^{*}Gr_{Fil}(H,\nabla).$$
	\end{theorem}
	\iffalse
	\begin{question} Let $M$ be a compact K\"ahler manifold and $Y\subset M$ be a SNCD. Take $\V$ be a PVHS on $X-Y$ with nilpotent local monodromy along $Y$ and denote by $(\mathcal{V},\nabla,Fil)$ the de Rham bundle over $M-Y$ associated to $\V$. Write $(\overline{\mathcal{V}},\overline{\nabla},\overline{Fil})$ to be the Deligne extension of $(\mathcal{V},\nabla,Fil)$ over $M$ (see Definition \ref{IC def C}). Since an one-periodic de Rham bundle $(H,\nabla,Fil)$ over $X_{\log}$ is the analogy of $(\overline{\mathcal{V}},\overline{\nabla},\overline{Fil})$ in positive characteristic, then it is natural to ask wether following equality of complexes holds or not \footnote{When $\V$ comes from geometry, the answer is yes, see Theorem \ref{inter zero}.}:
		\begin{eqnarray}\label{inter comm PVHS}Gr_{Fil}\Omega_{int}^{*}(\overline{\mathcal{V}},\overline\nabla)=\Omega_{int}^{*}Gr_{Fil}(\overline{\mathcal{V}},\overline\nabla).\end{eqnarray}
		Our method is available to prove \ref{inter comm PVHS}, but lemmas \ref{local}, \ref{bundle map} need new proofs in complex analytic setting. Lemma \ref{local} can be proved by using $SL_2$-orbit Theorem to find a sufficiently good local basis for $\overline{\mathcal{V}}$, see \cite{Zuc}, \cite{JYZ}. The problem is the interesting lemma \ref{bundle map}.
	\end{question}
	\fi
	To approach the problem, it is good to consider it in a general setting. Let $S$ be a noetherian regular scheme and $D\subset X$ be a reduced NCD relative to $S$. Let $(H,\nabla,Fil)$ be a de Rham module over $X_{\log}/S$. Set $(E=\bigoplus_iE^i,\bigoplus_i\theta^i)$ to be the graded Higgs module.
	\begin{lemma}\label{inclusion of intersection subcomplex}
		Notation as above. Then one has natural inclusions of complexes
		$$
		\Omega^*_{int}(E,\theta)\subset Gr_{Fil}\Omega_{int}^{*}(H,\nabla)\subset \Omega^{*}(E,\theta).
		$$
	\end{lemma}
	\begin{proof}
		Note the second inclusion is obvious for we have
		$$
		Gr_{Fil}\Omega_{int}^{*}(H,\nabla)\subset Gr_{Fil}\Omega^{*}(H,\nabla)=\Omega^*Gr_{Fil}(H,\nabla)=\Omega^*(E,\theta).
		$$
		To show the first inclusion, we notice first that it is an \'etale local problem and holds trivially outside $D$. So we may assume $\tilde X$ admits an admissible system of local coordinates $((\tilde t_1,\cdots,\tilde t_n),r)$ with $r>0$. Let $\nabla_i,\theta_i,1\leq i\leq n$ be the components of $\nabla,\theta$ with respect to the admissible system. By Proposition \ref{multiweight Higgs}, one has
		$$
		\Omega_{int}^{\bullet}(H,\nabla)=\bigoplus_{\beta\in P^\bullet}H_{w(\beta)}\otimes\beta
		$$
		and
		$$
		\Omega_{int}^{\bullet}(E,\theta)=\bigoplus_{\beta\in P^\bullet}E_{w(\beta)}\otimes\beta.
		$$
		As the filtration $Fil$ on $\Omega_{int}^{\bullet}(H,\nabla)$ is compatible with the direct decomposition, it is equivalent to show $E_{w(\beta)}\subset Gr_{Fil}H_{w(\beta)}$ for each $\beta\in P^\bullet$. Let $w(\beta)=w^I$ for some $I\subset \{1,\cdots,r\}$. Then by Griffiths transversality, we have for any fixed $a$ the following inclusion:
		\begin{eqnarray}\label{inclusion}
			\sum_{K\subset I}t_{I-K}\nabla_{K}Fil^{a+|K|}H&\subseteq& Fil^{a}\sum_{K\subset I}t_{I-K}\nabla_{K}H=Fil^{a}H_{w(\beta)}.
		\end{eqnarray}
		By the commutativity of the following diagram (with natural projections as vertical maps)
		\begin{eqnarray}\label{Griffiths transversality}\xymatrix{Fil^{a+|K|}H\ar[r]^-{\nabla_{K}}\ar[d]&Fil^{a}H\ar[d]\\
				E^{a+|K|}\ar[r]^-{\theta_{K}}&E^{a},}
		\end{eqnarray}
		the image in $E^a$ of the summand $t_{I-K}\nabla_{K}Fil^{a+|K|}H$ in the left hand side of the inclusion \eqref{inclusion}, is equal to $t_{I-K}\theta_{K}E^{a+|K|}$. Summing up over various $K$s, the inclusion \eqref{inclusion} leads to the inclusion $E^a_{w(\beta)}\subset Gr^a_{Fil}H_{w(\beta)}$. As $a$ is arbitrary, the first inclusion holds.
	\end{proof}
	In the proof of  Lemma \ref{inclusion of intersection subcomplex}, one sees clearly that if the inclusion \eqref{inclusion} becomes an equality, that is if the following condition
	\begin{eqnarray}\label{suff cond}
		Fil^{a}\sum_{K\subset I}t_{I-K}\nabla_{K}H&=&\sum_{K\subset I}t_{I-K}\nabla_{K}Fil^{a+|I|}H
	\end{eqnarray}
	holds for any $a$ and any $I\subset\{1,\cdots,r\}$, then $\Omega^*_{int}(E,\theta)= Gr_{Fil}\Omega_{int}^{*}(H,\nabla)$ holds on $X$. By Lemma \ref{periodicity stable under etale base change} and the faithfully flat descent, one immediately deduces Theorem \ref{inter comm} from the next
	\begin{proposition}\label{proposition suff cond}
		Let $S$ be the spectrum of a perfect field $k$ of characteristic $p$. Then \eqref{suff cond} holds when the following conditions are both satisfied:
		\begin{itemize}
			\item[(i)] $\tilde X_{\log}$ admits an admissible system of local coordinates $t$,
			\item[(ii)] $(H,\nabla,Fil)$ is a one-periodic de Rham bundle over $X_{\log}/k$.
		\end{itemize}
	\end{proposition}
	For each $i$, let us denote the residue of $\nabla$ along the component $D_i\subset D$ by
	$$
	\res_{D_i}\nabla\in \End_{\sO_{D_i}}(H|_{D_i}),
	$$
	and denote for an arbitrary nonempty set $I\subset \{1,\cdots,n\}$, the residue $\res_{D_I}\nabla$ of $\nabla$ along the strata $D_I:=\cap_{i\in I}D_i$ \footnote{For $I=\emptyset$, we set $D_{I}=X$.}
	$$
	\res_{D_I}\nabla:=\prod_{i\in I}(\res_{D_i}\nabla)|_{D_I}\in \End_{\sO_{D_I}}(H|_{D_I}).
	$$
	The following properties of the residues are essential ingredients in the proof of Proposition \ref{proposition suff cond}.
	\begin{lemma}\label{bundle map}
		Notation as in Proposition \ref{proposition suff cond}. Then the following hold:
		\begin{itemize}
			\item[(i)] $\res_{D_I}\nabla$ is strict, i.e. $\res_{D_I}\nabla(Fil^{a+|I|}H|_{D_I})=\mathrm{im}(\res_{D_I}\nabla)\cap Fil^aH|_{D_I}, \forall a$;
			\item[(ii)] for any $a$, $\res_{D_I}\nabla$ restricts to a bundle morphism
			$$\res_{D_I,a}\nabla: Fil^{a+|I|}H|_{D_I}\longrightarrow Fil^aH|_{D_I}.$$
			In other words, the kernel and the cokernel of the morphism $\res_{D_I,a}\nabla$ are locally free.
		\end{itemize}
	\end{lemma}
	\begin{proof}
		We may even assume $\tilde X=\Spec\ R$ affine. By definition, as $(H,\nabla,Fil)$ is one-periodic, there is an isomorphism of modules with integrable connection over $X$:
		$$
		\psi: C^{-1}Gr_{Fil}(H,\nabla)\cong (H,\nabla).
		$$
		Forgetting the connection part, it is indeed an isomorphism of $R$-modules $\psi: F^*Gr_{Fil}H\cong H$. So it is in fact a strict $p$-torsion object of the category $MF(R)$ in Ch. II c) \cite{Fa} \footnote{Strictly speaking, Faltings defines the category $MF(R)$ for $R$ smooth over $W(k)$. However a strict $p$-torsion object in the category is indeed defined over its mod $p$ reduction. Moreover the proof of Theorem 2.1 loc. cit. is reduced to strict $p$-torsion objects over $R/pR$. So it follows that the full subcategory $MF(R)$ of strict $p$-torsion objects and its consequent Theorem 2.1 only require $R$ be smooth over $k$.}. Thus, for each $I$, $\psi$ restricts to an isomorphism $F^*Gr_{Fil}H|_{D_I}\cong H|_{D_I}$, and hence defines $(H|_{D_I},Fil)$ a strict $p$-object in $MF(R_I)$ with $R_I$ the coordinate ring of $D_I$. As $\psi$ is horizontal, $\res_{D_i}$ is a morphism in the category $MF(R_i)$ for any $i$. Consequently, $\res_{D_I}$ is a morphism in $MF(R_I)$. Now (i) follows from \cite[Theorem 2.1 iii)]{Fa}, and (ii) follows from \cite[Corollary to Theorem 2.1, Theorem 2.1]{Fa}.
	\end{proof}
	Now we proceed to the proof of Proposition \ref{proposition suff cond} as follows.
	\begin{proof}
		If $I=\emptyset$, then both sides of the \eqref{suff cond} are equal to $Fil^aH$. Thus we may assume $I$ to be nonempty from the beginning. It suffices to show a section $s$ lying in $Fil^{a}\sum_{K\subset I}t_{I-K}\nabla_{K}H$ over some open affine subset of $X$ admits the following expansion:
		\begin{eqnarray}\label{expansion along D}
			s=\sum_{0\leq i\leq |I|}s_i,~s_i=\sum_{K\subset I, |K|=i}t_{I-K}\nabla_Ks_K,~s_K\in Fil^{a+|K|}H.
		\end{eqnarray}
		The above expansion will be attained, by applying consecutively the following
		\begin{claim}
			If a section $s\in Fil^{a}\sum_{K\subset I}t_{I-K}\nabla_{K}H$ vanishes along $D_K$ for any $K\subset I$ with $|K|=i+1$ for some $0\leq i\leq |I|$, then there exists some section of form $$s_{i}=\sum_{K\subset I, |K|=i}t_{I-K}\nabla_Ks_K,~s_K\in Fil^{a+|K|}H$$ such that the difference $s-s_{i}$ vanishes along $D_K$ for any $K\subset I$ with $|K|=i$.
		\end{claim}
		As the condition of the claim for $i=|I|$ is void, we shall obtain a section $s_{|I|}$ such that $s-s_{|I|}$ vanishes along the strata $D_{I}$. Then we continue downward, until we arrive at a section $s-s_{|I|}-\cdots-s_0$ vanishing along $D_{\emptyset}=X$, from which the expansion for $s$ is obtained.
		
		The defining ideal of $D$ is principal with generator $t_I$. Therefore, if a section $s\in H$ vanishes along all $D_i$s, it must be of form $s_0=t_Is'$ for some $s'\in H$. So the claim for $i=0$ holds. Now
		given an $s\in H$ as in the claim with $|I|>0$ and any $K_0\subset I$ with $|K_0|=i$, we aim to construct an $s_{K_0}\in Fil^{a+|K_0|}H$ such that $s|_{D_{K_0}}=t_{I-K_0}\nabla_{K_0}s_{K_0}$. Since $t_{I-K_0}$ vanishes along any $D_{K'}$ with $K'\subset I, K'\neq K_0, |K'|=i$, it follows that $s_{i}=\sum_{K\subset I, |K|=i}t_{I-K}\nabla_Ks_K$ is a solution as required. By condition, $s|_{D_{K_0}}$ vanishes along any $D_{K_0\cup \{j\}}$ with $j\in I-K_0$. So
		$$
		s|_{D_{K_0}}=(t_{I-K_0}s'_{K_0})|_{D_{K_0}},\quad s'_{K_0}\in Fil^aH.
		$$
		Notice that the image of the natural morphism $(\sum_{K\subset I}t_{I-K}\nabla_{K}H)|_{D_{K_0}}\rightarrow H|_{D_{K_0}}$ is contained in the image of
		$$
		\res_{D_{K_0}}\nabla: H|_{D_{K_0}}\to H|_{D_{K_0}}.
		$$
		It follows that
		$
		s\in Fil^aH\cap\mathrm{im}(\res_{D_{K_0}}\nabla).
		$
		\iffalse Therefore, by Lemma \ref{bundle map} (i), one has
		\begin{eqnarray}\label{s K0}
			s'_{K_0}\in
		\end{eqnarray}\fi
		Consider the following exact sequence of $\sO_{D_{K_0}}$-modules:
		$$
		Fil^{a+|K_0|}H|_{D_{K_0}}\stackrel{\res_{D_{K_0},a}\nabla}{\longrightarrow} Fil^{a}H|_{D_{K_0}}\stackrel{\pi_{K_0,a}}{\longrightarrow}\mathrm{coker}(\res_{D_{K_0},a}\nabla)\to 0.
		$$
		By Lemma \ref{bundle map} (i), it follows that
		$$
		\pi_{K_0,a}(s'_{K_0}|_{D_{K_0}})=0~\mathrm{on}~U\cap(D_{K_0}-\cup_{j\in I-K_0} D_j).
		$$
		On the other hand, since $D_{K_0}$ is integral and $\mathrm{coker}(\res_{D_{K_0},a}\nabla)$ is locally free by Lemma \ref{bundle map}, it follows that $\pi_{K_0,a}(s'_{K_0}|_{D_{K_0}})=0$. Therefore, there exists some section $s_{K_0}$ of $Fil^{a+|K_0|}H$ such that
		$$
		s'_{K_0}|_{D_{K_0}}=\res_{D_{K_0},a}\nabla(s_{K_0}|_{D_{K_0}}).
		$$
		Consequently, $s=t_{I-K_0}\nabla_{K_0}s_{K_0}$ along $D_{K_0}$ with $s_{K_0}\in Fil^{a+|K_0|}H$. The claim is proved. The proposition follows.	
	\end{proof}

\subsection{$E_1$-degeneration and vanishing theorem}\label{E_1 degeneration subsection}
	The data of a one-periodic de Rham bundle $(H,\nabla,Fil)$ is encoded in the following triangle:
	$$
	\xymatrix{  (H,\nabla) \ar[dr]^{Gr_{Fil}}
		&&  C^{-1}(E,\theta),\ar@/_2pc/[ll]_{\stackrel{\psi}{\cong}  }\\
		&Gr_{Fil}(H,\nabla)=(E,\theta).\ar[ur]^{C^{-1}}}
	$$
	The following $E_1$-degeneration theorem has a one-sentence proof.
	\begin{theorem}\label{E_1 deg}
		Assume additionally that $X$ is proper over $k$. Let $(H,\nabla,Fil)$ be a one-periodic de Rham bundle over $X_{\log}$ with the level of the filtration $l\leq p-1$. If $\dim X+l< p$, then the Hodge to de Rham spectral sequence associated to the intersection de Rham complex $\Omega^*_{int}(H,\nabla)$ degenerates at $E_1$.
	\end{theorem}
	\begin{proof}
		In the above triangle, applying the intersection adaptedness theorem (Theorem \ref{inter comm}) to the left down side and then the de Rham-Higgs comparison theorem (Theorem \ref{dec thm}) to the right down side, one obtains the following (in)equalities on the dimensions of finite dimensional $k$-vector spaces:
		\begin{eqnarray*}
			\dim Gr_{Fil}\HH^i(X,\Omega^*_{int}(H,\nabla))&\leq &\dim \HH^i(X,Gr_{Fil}\Omega^*_{int}(H,\nabla))\\
			&=&\dim \HH^i(X,\Omega^*_{int}Gr_{Fil}(H,\nabla))\\
			&=& \dim \HH^i(X,F_{X*}\Omega^*_{int}C^{-1}(E,\theta))\\
			&=& \dim \HH^i(X,F_{X*}\Omega^*_{int}(H,\nabla))\\
			&=&\dim \HH^i(X,\Omega^*_{int}(H,\nabla))\\
			&=& \dim Gr_{Fil}\HH^i(X,\Omega^*_{int}(H,\nabla)),
		\end{eqnarray*}
		which implies the first inequality is actually an equality and hence the theorem follows.
	\end{proof}
	It is useful to extend the intersection adaptedness theorem and hence the $E_1$-degeneration theorem to a general periodic de Rham bundle.
	\begin{theorem}\label{loc dec thm}
		Notation and assumption as Theorem \ref{inter comm}. Let $(H,\nabla,Fil)$ be a periodic de Rham bundle over $X_{\log}$. Then the following holds
		$$
		Gr_{Fil}\Omega^*_{int}(H,\nabla)=\Omega^*_{int}Gr_{Fil}(H,\nabla).
		$$
	\end{theorem}
	\begin{proof}
		Assume the period of $(H,\nabla,Fil)$ to be $f$ and set $\{(H_i,\nabla_i,Fil_i)\}$ to be a defining sequence of the $f$-periodicity of $(H,\nabla,Fil)$. One observes that the direct sum of de Rham bundles
		$$
		\bigoplus_{0\leq i\leq f-1}(H_i,\nabla_i,Fil_i)
		$$
		is a one-periodic de Rham bundle. It is clear that the grading functor $Gr_{Fil}$ and the intersection functor $\Omega^*_{int}$ commute with direct sum. Hence, by Theorem \ref{inter comm} we get that
		$$
		\bigoplus_{i=0}^{f-1}\Omega^*_{int}Gr_{Fil_i}(H_i,\nabla_i)=\bigoplus_{i=0}^{f-1}Gr_{Fil_i}\Omega^*_{int}(H_i,\nabla_i).
		$$
		By Lemma \ref{inclusion of intersection subcomplex}, we have the natural inclusion for each $i$:
		$$
		\Omega^*_{int}Gr_{Fil_i}(H_i,\nabla_i)\subset Gr_{Fil_i}\Omega^*_{int}(H_i,\nabla_i).
		$$
		Therefore, we have the equality for each $i$,
		$$
		\Omega^*_{int}Gr_{Fil_i}(H_i,\nabla_i)=Gr_{Fil_i}\Omega^*_{int}(H_i,\nabla_i).
		$$
		The $i=0$ case of the above equality is the claimed result.
	\end{proof}
	\begin{corollary}\label{E_1 degeneration}
		Assume additionally that $X$ is proper over $k$. Let $(H,\nabla, Fil)$ be a periodic de Rham bundle over $X_{\log}$. If $\dim(X)+\rank(H)\leq p$, then the Hodge to de Rham spectral sequence of the intersection de Rham complex $\Omega_{int}^*(H,\nabla)$ degenerates at $E_1$.
	\end{corollary}
	\begin{proof}
		We shall use the notation in the proof of Theorem \ref{loc dec thm}. The condition on rank guarantees the level of the filtration $Fil_i$ in the sequence is of $\leq p-1$. Then by Theorem \ref{loc dec thm}, the proof of Theorem \ref{E_1 deg} yields the following sequence of (in)equalities:
		$$\begin{array}{rcl}
			\dim \HH^i(X,\Omega^*_{int}(H_0,\nabla_0))&\leq&\dim\HH^i(X,Gr_{Fil_0}\Omega^*_{int}(H_0,\nabla_0))\\
			&=&\dim \HH^i(X,\Omega^*_{int}(H_1,\nabla_1))\\
			&\vdots&\\
			&\leq& \dim\HH^i(X,\Omega^*_{int}(H_{f},\nabla_{f})).\end{array}
		$$
		Since the last term is equal to the first term by the periodicity, the above inequalities are all equalities.  The first equality in above is the claim.
	\end{proof}
	\begin{remark}\label{E_1 for log}
		In Theorem \ref{E_1 deg} and Corollary \ref{E_1 degeneration}, the Hodge to de Rham spectral sequence associated to the de Rham complex $\Omega^*(H,\nabla)$ also degenerates at $E_1$. The above proofs works verbatim after replacing $\Omega^{*}_{int}$ with $\Omega^*$. This generalizes the $E_1$ degeneration theorem of Illusie attached to semistable families \cite{IL90} (see also \cite{Fa}, \cite{Ogus}).
	\end{remark}
	One of the most significant applications of $E_1$-degeneration is the Kodaira-type vanishing theorem (see \cite{EV}). The following vanishing result generalizes many known in literature.
	\begin{corollary}\label{vanishing result}
		Assume additionally that $X$ is projective over $k$. Let $(E,\theta)$ be a periodic Higgs bundle on $X_{\log}$. If the inequality $\dim(X)+\rank(E)\leq p$ holds, one has the following vanishing
		$$
		\HH^i(X,\Omega_{int}^*(E,\theta)\otimes L)=0, \quad i>\dim X
		$$
		for any ample line bundle $L$ on $X$.
	\end{corollary}
	\begin{proof}
		Assume the period of $(E,\theta)$ to be $f$.  Let $(E_i,\theta_i)$ be the Higgs terms of an $f$-periodic Higgs-de Rham flow with the initial term $(E_0,\theta_0)=(E,\theta)$. Theorem \ref{dec thm} and Theorem \ref{loc dec thm} together yield a sequence of (in)equalities
		$$\begin{array}{rcl}
			\dim \HH^i(\Omega^*_{int}(E_0,\theta_0)\otimes L)&= & \dim \HH^i(\Omega^*_{int}C^{-1}(E_0,\theta_0)\otimes L^p)\\
			&\leq & \dim \HH^i(Gr_{Fil}\Omega^*_{int}C^{-1}(E_0,\theta_0)\otimes L^p)\\
			&= & \dim \HH^i(\Omega^*_{int}(E_1,\theta_1)\otimes L^p)\\
			&\vdots&\\
			&\leq &\dim \HH^i(\Omega^*_{int}(E_f,\theta_f)\otimes L^{p^f}).
		\end{array}$$
		As $(E_f,\theta_f)\cong (E_0,\theta_0)$ by assumption, we obtain the following inequality
		$$
		\dim \HH^i(\Omega^*_{int}(E,\theta)\otimes L)\leq \dim \HH^i(\Omega^*_{int}(E,\theta)\otimes L^{p^f}).
		$$
		It follows that for any $m\in \N$,
		$$
		\dim \HH^i(\Omega^*_{int}(E,\theta)\otimes L)\leq \dim \HH^i(\Omega^*_{int}(E,\theta)\otimes L^{p^{mf}}).
		$$
		On the other hand, the Serre vanishing theorem implies that for $i>\dim X$,
		$$
		\dim \HH^i(\Omega^*_{int}(E,\theta)\otimes L^{p^{mf}})=0,~ m\gg1.
		$$
		Therefore $\dim \HH^i(\Omega^*_{int}(E,\theta)\otimes L)=0$ for $i>\dim X$ as claimed.
	\end{proof}

	Based on Theorem \ref{dec thm} and Theorem \ref{E_1 deg},  we are able to obtain a basic structure theorem of intersection cohomology groups attached to strict $p$-torsion logarithmic Fontaine modules. More generally, one may extend the result below for strict $p$-torsion logarithmic Fontaine modules with endomorphism structure $\F_{p^f}$. See \cite[\S2 Variant 2]{LSZ1} for the basic definition. We leave the necessary modification for this general case to interested readers.
	
	Recall that the category of strict $p$-torsion Fontaine-Laffaille modules consists of the following objects $(M, Fil, \phi)$, where $M$ is a finite dimensional $k$-vector space, $Fil$ is a finite decreasing filtration on $M$ satisyfing $Fil^0=M$, and $\phi$ is a $k$-linear isomorphism $F_k^*Gr_{Fil}M\to M$ (the category is denoted as $\widetilde{MF}^f_{tor}$ in \cite{FL}). We shall give a slight improvement of the following result of G. Faltings \cite[Theorem 4.1*]{Fa}\footnote{In loc. cit. the condition on the Hodge-Tate weight $w$ reads $w+\dim X\leq p-2$.}.
	\begin{theorem}[Faltings]\label{FL for log}
		Let $(H,\nabla,Fil,\Phi)$ be a strict $p$-torsion Fontaine module over $\tilde X_{\log}$ of Hodge-Tate weight $w$. If $w+\dim X\leq p-1$, then for each $i$ the cohomology group
		$$H_{DR}^i(X,(H,\nabla)):=\HH^i(X,\Omega^*(H,\nabla))$$
		is a strict $p$-torsion Fontaine-Laffaille module.
	\end{theorem}
	\begin{proof}
		By Proposition \ref{one-periodic} and $E_1$-degeneration (Remark \ref{E_1 for log}), one has the equalities of $k$-vector spaces:
		$$
		Gr_{Fil}\HH^i(X,\Omega^*(H,\nabla))=\HH^i(X, Gr_{Fil}\Omega^*(H,\nabla))=\HH^i(X,\Omega^*Gr_{Fil}(H,\nabla)).
		$$
		By flat base change, Theorem \ref{equivalence} (ii) and Proposition \ref{one-periodic} again, one has
		$$\begin{array}{rcl}
			F_k^*\HH^i(X,\Omega^*Gr_{Fil}(H,\nabla))&\cong &\HH^i(X', \Omega^*Gr_{Fil}(H,\nabla))\\
			&\cong&\HH^i(X', F_*\Omega^*(C^{-1}Gr_{Fil}(H,\nabla)))\\
			&\cong&\HH^i(X',F_*\Omega^*(H,\nabla))\\
			&\cong&\HH^i(X,\Omega^*(H,\nabla))
		\end{array}$$
		\iffalse
		Let $C^{-1}$ be the inverse Cartier transform of Ogus-Vologodsky. Then Proposition \ref{one-periodic} says that $\Phi$ induces an isomorphism of modules with integrable connection
		$$
		\tilde \Phi: C^{-1}(\pi^*Gr_{Fil}(H,\nabla))\cong (H,\nabla).
		$$
		It induces then an isomorphism of de Rham complexes (which is denoted again by $\tilde \Phi$)
		$$
		\Omega^*C^{-1}(\pi^*Gr_{Fil}(H,\nabla))\cong \Omega^*(H,\nabla).
		$$
		By the de Rham-Higgs comparison theorem of Ogus-Vologodsky, one has
		$$
		\HH^i(X', \Omega^*(\pi^*Gr_{Fil}(H,\nabla)))=\HH^i(X', F_*\Omega^*(C^{-1}\pi^*Gr_{Fil}(H,\nabla))),
		$$\fi
		The above computations show that $H^i_{DR}(X_{\log},(H,\nabla))$ is naturally a strict $p$-torsion Fontaine-Laffaille module.
	\end{proof}
	
	The natural inclusion $\Omega_{int}^*(H,\nabla)\subset \Omega^*(H,\nabla)$ induces natural morphisms for their cohomology groups. We show further that
	\begin{corollary}\label{FL module}
		Let $(H,\nabla,Fil,\Phi)$ be as in Theorem \ref{FL for log}. Then for each $i$, the intersection cohomology group associated to $(H,\nabla)$ (Definition \ref{intersection cohomology group})
		$$IH^i(X,(H,\nabla)):=\HH^i(X,\Omega^*_{int}(H,\nabla))$$
		admits the structure of a strict $p$-torsion Fontaine-Laffaille module such that the natural morphism $$IH^i(X,(H,\nabla))\to H_{DR}^i(X,(H,\nabla))$$ is a morphism of strict $p$-torsion Fontaine-Laffaille modules.
	\end{corollary}
	\begin{proof}
		In the proof of Theorem \ref{FL for log}, by replacing the equality $$Gr_{Fil}\Omega^*(H,\nabla)=\Omega^*Gr_{Fil}(H,\nabla)$$ with the intersection adaptedness theorem (Theorem \ref{inter comm}),  Theorem \ref{equivalence} (ii) and the $E_1$ degeneration theorem with their corresponding intersection versions (Theorem \ref{dec thm}, Theorem \ref{E_1 deg}), one obtains the same conclusion for the intersection cohomology groups.  \iffalse Since the intersection condition is clearly preserved by the isomorphism $\tilde \Phi$ of complexes, the natural morphism $IH^i(X,H)\to H^i(X^0,H)$ is indeed a morphism of strict $p$-torsion Fontaine-Laffaille modules, as claimed.\fi
		The natural morphism
		$$IH^i(X_{\log},(H,\nabla))\to H_{DR}^i(X_{\log},(H,\nabla))$$
		induced by $\Omega_{int}^*(H,\nabla)\subset \Omega^*(H,\nabla)$ is obviously a morphism of strict $p$-torsion Fontaine-Laffaille modules, since the filtration on $\Omega_{int}^*(H,\nabla)$ is induced from the one on $\Omega^*(H,\nabla)$.
	\end{proof}
	\iffalse
	\begin{corollary}\label{ES iso}
		Let $(H,\nabla, Fil, \Phi)$ be an object in $MF^{\nabla}_{[0,p-1]}(X_{log}/k)$. Then the Hodge to de Rham spectral sequence associated to the intersection de Rham complex of $(H,\nabla)$ degenerates at $E_1$. Consequently, one has the following equalities
		$$
		Gr_{Fil}\HH^{*}(X,\Omega_{int}^{*}(H,\nabla))= \HH^{*}(X,Gr_{Fil}\Omega_{int}^{*}(H,\nabla))= \HH^{*}(X,\Omega_{int}^{*}Gr_{Fil}(H,\nabla)).
		$$
	\end{corollary}
	\begin{proof}
		It follows from Proposition \ref{one-periodic} and Theorem \ref{E_1 deg}.
	\end{proof}
	\fi

	\section{Applications}
	In this section, $X$ is an $n$-dimensional smooth projective variety over $\C$ and $D\subset X$ a reduced normal crossing divisor. Set $X^0:=X-D$. The following fundamental result in Hodge theory was first established by Zucker \cite{Zuc} in the curve case and by Cattani-Kaplan-Schmid \cite{CKS} and independently by Kashiwara-Kawai \cite{KK} in general.
	\begin{theorem}\label{pure HS}
		Let $\V$ be a variation of Hodge structure of weight $w$ over $X^{0,an}$. Then the filtered intersection complex $(IC(X^{an},\V_{\C}),Fil)$ is a Hodge complex inducing a canonical pure Hodge structure of weight $w+m$ on the intersection cohomology $IH^m(X^{an},\V_\C)$.
	\end{theorem}
	Among other things, the theory of $L^2$-integrable harmonic forms plays a pivotal role in the proof of Theorem \ref{pure HS}. An algebraic consequence of this transcendental result is the $E_1$-degeneration of the spectral sequence:
	\begin{eqnarray}\label{Appli den}\quad\quad E_{1}^{p,q}=\mathbb{H}^{p+q}(X^{an},Gr^{p}_{Fil}IC(X^{an},\V_\C))\Rightarrow \mathbb{H}^{p+q}(X^{an},IC(X^{an},\V_\C)).
	\end{eqnarray}
	The major aim of this section is to provide a characteristic $p$ proof of this fact for those $\V$s coming from geometry.
	\begin{definition}
		A semistable family $f:Y\to X$ over $X$ is a projective morphism, smooth over $X^0$, with $Y$ smooth projective and $E:=f^{*}D\subset Y$ a reduced normal crossing divisor. Set $Y^0:=Y-E$ and $f^0:=f|_{Y^0}$. We  call $\V=R^wf^{0,an}_*\Q_{Y^{0,an}}$ for some $w\geq 0$ a VHS of geometric origin.
	\end{definition}
	We are going to prove the following
	\begin{theorem}\label{app main theo} Let $\V$ be a VHS of geometric origin over $X^{0,an}$. Then the spectral sequence \eqref{Appli den} degenerates at $E_1$.
	\end{theorem}
	By the local monodromy theorem of Landman \cite{Landman}, we know the local monodromies of $\V$ of geometric origin are unipotent. Let us recall briefly the definition of Deligne extension in this setting as well as the consequent definition of intersection complex attached to $\V_{\C}$. We shall use \cite[\S8]{CEGT} as the major reference.
	
	It is enough to look at the local situation. So we set
	$$
	U^{an}=\Delta^{n},D^{an}=(t_1\cdots t_r=0),U^{0,an}=(\Delta^{*})^{r}\times\Delta^{n-r},1\leq r\leq n,$$
	where $\Delta\subset\mathbb{C}$ is the unit disk, $\Delta^{*}=\Delta-\{0\}$ and $t_1,\cdots,t_n$ are the standard coordinates on $U^{an}$. Let $j:U^{0,an}\hookrightarrow U^{an}$ be the natural inclusion, $O\in U^{0,an}$ a reference point. Let $\V$ be any VHS over $U^{0,an}$. Set $V:=\V_{O}$, and $(\mathcal{V}^{an},\nabla^{an},Fil^{an})$ the associated analytic de Rham bundle to $\V$. Let $T_{1},\cdots,T_{r}\in \mathrm{End}_{\mathbb{Q}}(V)$ be the local monodromies along $r$-components of $D$.
	\begin{definition}[\S8.3.3.1 \cite{CEGT}] \label{IC def C}Notation as above. Let $\overline{\mathcal{V}^{an}}\subset j_*\mathcal{V}^{an}$ be the $\mathcal{O}_{U^{an}}$-submodule generated by the following sections
		$$
		\tilde{v}=\mathrm{exp}(-\frac{1}{2\pi \sqrt[]{-1}}\sum_{1\leq i\leq r}N_i\log t_i)v\in\Gamma(X^{an},j_*\mathcal{V}^{an}),~v\in V
		$$
		where $N_i=\log T_i,1\leq i\leq r$. Let $(\overline{\nabla^{an}},\overline{Fil^{an}})$ be the natural extension of $(\nabla^{an},Fil^{an})$ to $\overline{\mathcal{V}^{an}}$, defined as follows
		$$
		\overline{\nabla}\tilde{v}=\sum_{1\leq i\leq r}-\frac{1}{2\pi \sqrt[]{-1}}\widetilde{N_iv}\otimes d\log t_i,~s\in\overline{Fil}^{p}\overline{\mathcal{V}}\Leftrightarrow s\in Fil^{p}\mathcal{V}.
		$$
		The triple $(\overline{\mathcal{V}^{an}},\overline{\nabla^{an}},\overline{Fil^{an}})$ is called the Deligne extension of $(\mathcal{V}^{an},\nabla^{an},Fil^{an})$ over $U^{an}$.
	\end{definition}
	Let $D_I=\cap_{i\in I}D_i, I\in\{1,\cdots,r\}$ ($D_{\emptyset}=X$), and let $D^{*}_{I}=D_{I}-\cap_{I\subsetneqq J}D_{J}$. It is easy to see that for any point $x\in U^{an}$, there exists a unique $I\subset\{1,\cdots,r\}$ such that $x\in D^{*}_I$.
	\begin{definition}[Definition 8.2.23 \cite{CEGT}]
		Notation as above. The intersection complex $IC(U^{an},\V_\C)$ is the subcomplex of  $\Omega^{*}(\overline{\mathcal{V}^{an}},\overline{\nabla^{an}})$ whose stalk at a point $x\in D^{*}_I$ is defined as an $\Omega^{\bullet}_{U^{an},x}$-submodule, generated by sections
		$$
		\tilde{v}\wedge_{j\in J}d\log t_j,~v\in N_{J}V,~N_J=\prod_{j\in J}N_j~(N_{\emptyset}=Id),~J\subset I.
		$$
		The filtration $Fil$ on $IC(U^{an},\V_\C)$ is induced from the filtration on $\Omega^{*}(\overline{\mathcal{V}},\overline{\nabla})$.
	\end{definition}
	For a VHS $\V$ over $X^{0,an}$ with unipotent local monodromies around $D$, the intersection complex $IC(X^{an},\V_{\C})$ is obtained by patching various local intersection complexes $IC(U^{an},\V_\C)$ together. On the other hand, let $(\overline{\mathcal{V}^{an}},\overline{\nabla^{an}},\overline{Fil^{an}})$ be the Deligne extension of the analytic de Rham bundle $(\sV^{an},\nabla^{an},Fil^{an})$. By GAGA, the Deligne extension over $X$ is indeed algebraic. That is, there is an algebraic de Rham bundle $(\overline\sV,\overline \nabla,\overline{Fil})$ whose analytification is $(\overline{\mathcal{V}^{an}},\overline{\nabla^{an}},\overline{Fil^{an}})$. So we get the intersection de Rham complex $\Omega^*_{int}(\overline{\mathcal{V}},\overline{\nabla})$ which is naturally filtered.
	\begin{lemma}\label{two def equ}
		Notation as above. Then the analytification of $\Omega^*_{int}(\overline{\mathcal{V}},\overline{\nabla})$ coincides with the intersection complex $IC(X^{an},\V_\C)$ as filtered complex.
	\end{lemma}
	\begin{proof}
		Write $\Omega^*_{int}(\overline{\mathcal{V}},\overline{\nabla})^{an}$ for the analytification. Both filtered complexes in question are filtered subcomplexes of $\Omega^*(\overline{\mathcal{V}},\overline{\nabla})^{an}$. So it suffices to show
		$$
		IC(X^{an},\V_\C)=\Omega^{*}_{int}(\overline{\mathcal{V}},\overline{\nabla})^{an}.
		$$
		The problem is analytically local. So we may assume that we are in the situation of Definition \ref{IC def C}. Write
		$$
		\overline{\nabla}=\sum_{i=1}^r\overline{\nabla}_i\otimes d\log t_i+\sum_{i=r+1}^n\overline{\nabla}_i\otimes dt_i.
		$$
		Notice that for any $v\in V$,
		$$
		\widetilde{N_iv}=-2\pi\sqrt[]{-1}\overline{\nabla}_i\tilde{v},\quad i\in\{1,\cdots,r\}.
		$$
		It follows that for any $J\subset\{1,\cdots,r\}$, we have
		$$
		\widetilde{N_Jv}=(-2\pi\sqrt[]{-1})^{|J|}\overline{\nabla}_{J}\tilde{v},v\in V.
		$$
		So the stalk of $IC(X^{an},\V_\C)$ at $x\in D^*_{I},I\subset\{1,\cdots,r\}$ is generated by sections
		$$\overline{\nabla}_{J}\tilde{v}\wedge_{j\in J}d\log t_j,~v\in V,~J\subset I.$$
		As over some analytically open neighborhood of $x$, any $t_i,i\in\{1,\cdots,r\}-I$ becomes a unit, we may replace the above set of generators by the following set of sections
		$$
		\overline{\nabla}_{J}\tilde{v}\wedge_{j\in J}d\log t_j,~v\in V,~J\subset\{1,\cdots,r\}.
		$$
		A comparison of the above set of generators with the one given in Lemma \ref{independence of special local coord} yields the equality $IC(X^{an},\V_\C)=\Omega^{*}_{int}(\overline{\mathcal{V}},\overline{\nabla})^{an}$ as claimed.
	\end{proof}
	The following lemma is well-known.
	\begin{lemma}\label{Deligne extension vs log Gauss-Manin}
		Let $f:(Y,E)\to (X,D)$ be a semistable family. Let
		$$
		(H,\nabla,Fil):=(R^wf_{*}\Omega^{*}_{Y/X}(\log E/D),\nabla^{GM},Fil)
		$$
		be the logarithmic Gau{\ss}-Manin system of degree $w$ associated to $f$. Then
		$$
		(H,\nabla,Fil)=(\overline{\mathcal{V}},\overline{\nabla},\overline{Fil}).
		$$	
	\end{lemma}
	\begin{proof}
		See \cite[Remark 11.5, Corollary 11.19]{steen}.
	\end{proof}
	
	It is clear that Lemmata \ref{two def equ} and \ref{Deligne extension vs log Gauss-Manin} reduce the proof of Theorem \ref{app main theo} to the following statement.
	\begin{theorem}\label{E1 thm in char zero}
		Notation as in Lemma \ref{Deligne extension vs log Gauss-Manin}. Then, the spectral sequence of the filtered complex $(\Omega^*_{int}(H,\nabla),Fil)$ degenerates at $E_1$.
	\end{theorem}
	We shall use the spread-out technique to reduce Theorem \ref{E1 thm in char zero} to its corresponding statement in characteristic $p$ viz. Theorem \ref{E_1 deg}. The same idea applies to showing the following intersection adaptedness theorem in characteristic zero.
	\begin{theorem}\label{inter zero}
		Notation as in Theorem \ref{E1 thm in char zero}. Then the following equality holds:
		$$
		\Omega^*_{int}Gr_{Fil}(H,\nabla)=Gr_{Fil}\Omega^*_{int}(H,\nabla).
		$$
	\end{theorem}
	There is an important consequence of the previous results on the pure Hodge structure of the intersection cohomology groups of $\V$. The graded Higgs bundle $Gr_{Fil}(H,\nabla)$ reads
	$$
	(E=\bigoplus_{i+j=w} E^{i,j}=R^jf_{*}\Omega^i_{Y/X}(\log E/D),\theta=\bigoplus_{i+j=n}\theta^{i,j}),
	$$
	with $\theta^{i,j}: E^{i,j}\to E^{i-1,j+1}\otimes \Omega_X(\log D)$ the $i$-th Kodaira-Spencer morphism. Note that the Higgs complex $\Omega^*(E,\theta)$ is naturally graded:
	$$
	\Omega^*(E,\theta)=\bigoplus_{P=0}^{w+n}\Omega^*_{P}(E,\theta),~\Omega^i_P(E,\theta):=E^{P-i,w-P+i}\otimes \Omega_X^i(\log D).
	$$
	It induces a grading on $\Omega_{int}^*(E,\theta)$:
	$$
	\Omega_{int}^*(E,\theta)=\bigoplus_{P=0}^{w+n}\Omega^*_{int,P}(E,\theta), \  \Omega^*_{int,P}(E,\theta):=\Omega^*_{int}(E,\theta)\cap \Omega^*_{P}(E,\theta).
	$$
	The following result is a vast generalization of the complex Eichler-Shimura isomorphism for modular curves (see \cite[\S12]{Zuc}).
	\begin{corollary}
		Let $f: (Y,E)\to (X,D)$ be a semistable family over $\C$ and $\V=R^nf^{0,an}_*\Q_{Y^{0,an}}$. Then, the Hodge $(P,Q)$-component of the pure Hodge structure of weight $w+m$ on $IH^m(X^{an},\V)$ as given in Theorem \ref{pure HS} is naturally isomorphic to $\HH^{m}(X,\Omega^*_{int,P}(E,\theta))$.
	\end{corollary}
	\begin{proof}
		By Kashiwara-Kawai \cite{KK86, KK}, the middle perverse sheaf $\V_{\C}^{\pi}$ is quasi-isomorphic to the intersection complex $IC(X^{an},\V_{\C})$. Thus, by Lemma \ref{two def equ}, Theorems \ref{E1 thm in char zero}-\ref{inter zero} and the GAGA principle, it follows that
		\begin{eqnarray*}
			IH^{P,Q}(X^{an},\V_{\C})&\cong& Gr^P_{Fil}\HH^m(X^{an},\V_{\C}^{\pi})\\
			&\cong& Gr^P_{Fil}\HH^m(X^{an},IC(X^{an},\V_{\C}))\\
			&\cong& Gr^P_{Fil}\HH^m(X,\Omega^{*}_{int}(H,\nabla))\\
			&\cong& \HH^m(X,Gr^P_{Fil}\Omega^{*}_{int}(H,\nabla))\\
			&\cong& \HH^m(X,\Omega^{*}_{int}Gr^P_{Fil}(H,\nabla))\\
			&\cong& \HH^m(X,\Omega^{*}_{int,P}(E,\theta)).
		\end{eqnarray*}

	\end{proof}

	In the same vein, we get an algebraic proof of the following Kodaira type vanishing theorem for a VHS of geometric origin due to Saito (\cite[Proposition 2.33]{Saito}).
	\begin{corollary}
		Notation as Theorem \ref{inter zero}. Then, for any ample line bundle $L$ over $X$, it holds that
		$$
		\HH^i(X,\Omega^*_{int}(E,\theta)\otimes L)=0, \quad i>\dim X.
		$$
	\end{corollary}
	\begin{proof}
		Use the spread-out technique (as we shall explain in detail below). The result follows from Corollary \ref{vanishing result}.
	\end{proof}

	Now we proceed to the proofs of Theorems \ref{E1 thm in char zero}-\ref{inter zero}.
	\begin{proof}
		Write $\C$ as the inductive limit of its finitely generated sub $\Z$-algebras. Then by the standard argument (\cite[8, 11.2, 17.7]{EGA IV} ), there exists a sub $\Z$-algebra $A\subset \C$ of finite type and a semistable family $\mathfrak{f}: (\mathfrak Y, \mathfrak {E})\to (\mathfrak X, \mathfrak D)$ defined over $S=\Spec(A)$ such that $S$ is integral and regular, $\alpha: \mathfrak X\to S$ is smooth and projective, and $f$ is the base change of $\mathfrak f$ via $\Spec(\C)\to S$. Let $(\mathfrak{H},\nabla,\mathfrak{Fil})$ be the logarithmic Gau{\ss}-Manin system of degree $w$ associated to $\mathfrak{f}$, $(\mathfrak{E},\theta)=Gr_{\mathfrak{Fil}}(\mathfrak{H},\nabla)$. We claim that, after shrinking $S$ if necessary, the morphism $\mathfrak f$ satisfies the following additional properties:
		\begin{itemize}
			\item[(i)] For any closed point $s\in S$,
			$$\mathrm{char}(k(s))>N:=\mathrm{dim}(Y/X)+\mathrm{dim}(X)+\mathrm{rank}(H)+w.
			$$
			\item[(ii)] For any $i\geq 0$, the $\mathcal{O}_S$-modules $\R^i\alpha_*\Omega^*_{int}(\mathfrak H,\mathfrak \nabla)$ and $\R^i\alpha_*\Omega^*_{int}(\mathfrak E,\mathfrak \theta)$ are locally free of finite rank. The quotient $\sO_X$-modules $\Omega^\bullet(\mathfrak H,\mathfrak \nabla)/\Omega^\bullet_{int}(\mathfrak H,\mathfrak \nabla)$ and $\Omega^\bullet(\mathfrak E,\mathfrak \theta)/\Omega^\bullet_{int}(\mathfrak E,\mathfrak \theta)$ are flat over $S$.
			\item[(iii)] Let $S'\rightarrow S$ be the morphism given by either $\Spec(\C)\to S$ or $\Spec(k(s))\to S$ for a closed point $s\in S$. Let $f':(Y',E')\to(X',D')$ be the base change of $\mathfrak{f}$ via $S'\rightarrow S$. Then the pullbacks of $\Omega^*_{int}(\mathfrak{H},\nabla),\Omega^*_{int}(\mathfrak{E},\theta)$ and $Gr_{\mathfrak{Fil}}\Omega^*_{int}(\mathfrak{H},\nabla)$ to $X'$ are isomorphic to $\Omega^*_{int}(H',\nabla'),\Omega^*_{int}(E',\theta')$ and $Gr_{Fil'}\Omega^*_{int}(H',\nabla')$ respectively, where $(H',\nabla',Fil')$ is the logarithmic Gau{\ss}-Manin system of degree $w$ associated to $f'$ and $(E',\theta')=Gr_{Fil'}(H',\nabla')$.
			\item[(iv)] $Gr_{\mathfrak{Fil}}\Omega^*_{int}(\mathfrak{H},\nabla)=\Omega^*_{int}(\mathfrak{E},\theta)$.
			\item[(v)] Let $f_s$ be the fiber of $f$ over a closed point $s\in S$. Let $(H_s,\nabla_s,Fil_s)$ the logarithmic Gau{\ss}-Manin system of degree $w$ associated to $f_s$. Then $(H_s,\nabla_s,Fil_s)$ is one-periodic. \footnote{The one-periodicity refers to any $W_2(k(s))$-lifting of $X_s$, which is induced by a closed immersion $\Spec(W_2(k(s)))\to S$.}
		\end{itemize}
		Let us deduce first the theorems by assuming the claim. Note that Theorem \ref{inter zero} follows immediately by the properties (iii) and (iv). Set $(E_s,\theta_s)=Gr_{Fil_s}(H_s,\nabla_s)$. Then by Proposition \ref{cohomology and base change} and Theorem \ref{E_1 deg}, the following sequence of equalities hold, from which Theorem \ref{E1 thm in char zero} follows:
		\begin{eqnarray*}
			\dim_{\C} \HH^m(X, \Omega^*_{int}(H,\nabla))&=&\dim_{\C} R^m\alpha_*(\Omega^*_{int}(\mathfrak H,\mathfrak \nabla))\otimes_{A}\C\\
			&=&\dim_{k(s)}R^m\alpha_*(\Omega^*_{int}(\mathfrak H,\mathfrak \nabla))\otimes_Ak(s)\\
			&=&\dim_{k(s)} \HH^m(X_s, \Omega^*_{int}(H_s,\nabla_s))\\
			&=&\dim_{k(s)} \HH^m(X_s, \Omega^*_{int}(E_s,\theta_s))\\
			&=&\dim_{k(s)}R^m\alpha_*(\Omega^*_{int}(\mathfrak E,\mathfrak \theta))\otimes_Ak(s)\\
			&=&\dim_{\C} R^m\alpha_*(\Omega^*_{int}(\mathfrak E,\mathfrak \theta))\otimes_{A}\C\\
			&=&\dim_{\C} \HH^m(X, \Omega^*_{int}(E,\theta)).
		\end{eqnarray*}
		Among the five properties of the claim, (v) is the content of Proposition \ref{one-periodic}. (i) is obvious. The first part of (ii) follows from the finiteness theorem of Grothendieck which implies that  $\R^i\alpha_*\Omega^*_{int}(\mathfrak H,\mathfrak \nabla)$ and $\R^i\alpha_*\Omega^*_{int}(\mathfrak E,\mathfrak \theta)$ are coherent $\sO_S$-modules and hence locally free over some nonempty open subset of $S$. The second part of (ii) is the consequence of Theorem of generic flatness (\cite[Th\'{e}or\`{e}me 6.9.1]{EGA IV}). This step is necessary, since components of the intersection de Rham/Higgs complex are
		\emph{not} locally free in general (though $H$ are $E$ are locally free), which is in contrast with the logarithmic de Rham/Higgs complex.

		By Deligne-Illusie \cite{DI}, Illusie \cite{IL02}, for any $i,j$, $$R^{j}\mathfrak{f}_{*}\Omega^{i}_{\mathfrak{Y}/\mathfrak{X}}(\log \mathfrak{E}/\mathfrak{D}),~ R^{i+j}\mathfrak{f}_{*}\Omega^{*}_{\mathfrak{Y}/\mathfrak{X}}(\log \mathfrak{E}/\mathfrak{D})$$ are locally free of finite type, and the spectral sequence
		\begin{eqnarray}\label{int E1}E^{i,j}_1=R^{j}\mathfrak{f}_{*}\Omega^{i}_{\mathfrak{Y}/\mathfrak{X}}(\log \mathfrak{E}/\mathfrak{D})\Rightarrow R^{i+j}\mathfrak{f}_{*}\Omega^{*}_{\mathfrak{Y}/\mathfrak{X}}(\log \mathfrak{E}/\mathfrak{D})\end{eqnarray}
		degenerates at $E_1$. The corresponding statements of (iii) for $\Omega^*(\mathfrak H, \nabla)$ and $\Omega^*(\mathfrak E, \theta)$ follow from the above fact whose proofs are standard.  The point $\Spec(\C)\to S$ is the geometric generic point of $S$. So by Proposition \ref{cohomology and base change} (i), (iii) for the base change via $\Spec(\C)\to S$ follows. Using the $S$-flatness condition in (ii), one obtains (iii) for a closed point $s\in S$ by Proposition \ref{cohomology and base change} (ii).

		By Lemma \ref{inclusion of intersection subcomplex} and by Corollary \ref{multiweight Higgs}, we have the inclusion of coherent $\sO_{\mathfrak X}$-modules:
		$$\Omega^\bullet_{int}Gr_{\mathfrak{Fil}}(\mathfrak{H},\nabla)\subset Gr_{\mathfrak{Fil}}\Omega^\bullet_{int}(\mathfrak{H},\nabla).$$ Let $\mathfrak{Q}^\bullet$ be the cokernel of the inclusion. By generic flatness, we may assume $\mathfrak{Q}^{\bullet}$ flat over $S$. As $\mathfrak{Q}^{\bullet}$ is coherent and $f$ is proper, it suffices to show  $\mathfrak{Q}^{\bullet}\otimes k(s)=0$ for any closed point $s\in S$. By the base change property (iii) and Theorem \ref{inter comm} using (v), it follows that $\mathfrak{Q}^{\bullet}\otimes k(s)=0$. Thus (iv) follows.
	\end{proof}

\section{Appendix}
In this appendix, we show the existence of an $\infty$-homotopy as claimed in Theorem \ref{infty homotopy}. The whole section is divided into three parts.

\subsection{Higgs-de Rham ring}
Let $R$ be any commutative ring with identity. Let $A$ be the polynomial algebra over $R$ with the following three types of indeterminate:
\begin{itemize}
\item $\theta _l, ~l=1,2,\cdots$;
\item $e_I,~I\subset \mathbb \Z_{>0},|I|<\infty$;
\item $h_{k,l},~k,l=1,2,\cdots$.
\end{itemize}
Let $\mathfrak a$ be the ideal generated by indeterminate of the second type, and let $\bar A$ be the quotient ring $A/\mathfrak a^2$. Let $M$ be the free $\bar A$-module generated by
\begin{itemize}
\item $\zeta_{k,l},k=0,1,~\cdots,~l=1,2,\cdots$.
\end{itemize}
Recall that the exterior algebra $B_0:=\bigwedge_{\bar A}(M)$ is defined to be the quotient of the tensor algebra $T_{\bar A}(M)$ by the two-sided ideal generated by all expressions $x\otimes x$ for $x\in M$. So $B_0=\bigoplus_n\bigwedge^n(M)$ is a skew commutative graded $\bar A$-algebra.
\begin{definition}
Let $\mathcal I\subset B_0$ be the two-sided ideal generated by $\{\theta_i,e_I,h_{k,l},\zeta_{k,l}\}$, and for $s\in \N$,
$\mathcal J_s$ be the two-sided ideal generated by $\theta_i,e_I,h_{k,l},\zeta_{k,l}$ with $i,|I|,k+l\geq s$. Let $B_1$ to be the completion of $B_0$ with respect to the decreasing family of two-sided ideals $\{\mathcal I^s+\mathcal J^s\}_{s}$. Finally, we define
 $$B_{\mathrm{HdR}}:=B_1[\theta_i^{-1},i\geq1].$$
 We call it the Higgs-de Rham ring over $R$. To emphasize that this ring is defined over $R$, we add a superscript $R$, i.e. $B^R_{\mathrm{HdR}}$. When $R=\mathbb Q$, we simply write it as $B_{\mathrm{HdR}}$.
\end{definition}
 For each positive number $m$, let $\mathfrak b^R_m$ be the closure of the two-sided ideal  generated by monomials in $h_{k,l}, \zeta_{k,l}$ with total power $\geq m$. The quotient ring $B^R_{\mathrm{HdR},m}:=B^R_{\mathrm{HdR}}/\mathfrak{b}^R_m$ will be important for later use. Note that there is a natural decomposition of $R$-modules
\begin{eqnarray}\label{decomposition B}
B^R_{\mathrm{HdR}}=B^R_{\mathrm{HdR},m}\bigoplus\mathfrak{b}_m^R.
\end{eqnarray}
In fact, any element in $B^R_{\mathrm{HdR}}$ can be uniquely expressed as the sum of an  $R$-linear form combination of monomials in $\theta_l,e_I,h_{k,l},\zeta_{k,l}$ such that the total power in $h_{k,l},\zeta_{k,l}$ is $<m$ and an  $R$-linear form combination of monomials in $\theta_l,e_I,h_{k,l},\zeta_{k,l}$ such that the total power in $h_{k,l},\zeta_{k,l}$ is $\geq m$. Set $B^R_{\mathrm{HdR},f}:=\bigcup_mB^R_{\mathrm{HdR},m}$.

\subsection{An initial value problem and its solution}
In this subsection, our discussion is over $\mathbb Q$.
Let $d: B_0\to B_0$ be the additive map determined by the following rules:
$$
d\theta_i=de_I=d\zeta_{k,l}=0,~dh_{k,l}=\zeta_{k,l}-\zeta_{k-1,l},~d(xy)=(dx)y+(-1)^{n}xdy,
$$
for $x\in \bigwedge^n(M)$. Clearly, $d$ is continuous with the topology defined by $\{\mathcal I^s+\mathcal J^s\}_{s}$. Thus it extends uniquely an additive continuous operator on $B_1$, which is also denoted by $d$. We regard the element  $\sum_{i=1}^\infty\theta_i\zeta_{0,i}$ as an operator on $B_1$ by left multiplication. Set
$$
\nabla=d+\sum_{i=1}^\infty\theta_i\zeta_{0,i}.
$$
Clearly, $\nabla^2=0$.

Let $\Theta: B_0\to B_0$ be the differential operator on $B_0$ determined by
$$
\Theta\theta_i=\Theta h_{k,l}=\Theta \zeta_{k,l}=0,
$$
$\Theta e_{\emptyset}=0$, and for $I\neq \emptyset$,
$$
\Theta e_I=\sum_{1\leq k\leq s}(-1)^{k-1}e_{\{\cdots,\hat{i_k},\cdots\}},~I=\{i_1,\cdots,i_s\},~ i_1<\cdots<i_s.
$$
It is easy to verify $\Theta^2=0$. Again, $\Theta$ extends uniquely an additive continuous operator over $B_1$, and it satisfies $\Theta^2=0$.
For each $s\in \N$, we define an additive continuous operator $\mathrm{Shift}_s$ as follows: For $s=0$, $\mathrm{Shift}_0$ is determined by
$$
\theta_i\mapsto\theta_i,~e_I\mapsto e_I,~h_{k,l}\mapsto h_{k+1,l},~\zeta_{k,l}\mapsto\zeta_{k+1,l}.
$$
For $s>0$, $\mathrm{Shift}_s$ is determined by
$$
\theta_l\mapsto\theta_l,~e_I\mapsto e_I,~h_{k,l}\mapsto\left\{\begin{matrix}h_{k,l},~k<s\\
 h_{s,l}+h_{s+1,l},~k=s\\
 h_{k+1,l},~k>s
\end{matrix}\right.,~\zeta_{k,l}\mapsto\left\{\begin{matrix}
\zeta_{k,l},~k<s\\
\zeta_{k+1,l},~k\geq s.
\end{matrix}\right.
$$
 We linearly extend $\nabla, \Theta, \mathrm{Shift}_s, s\in \N$ to $B_{\mathrm{HdR}}$. The symbol $\mathrm{exp}(\sum_{l=0}^\infty\theta_l h_{1,l})$ is a well-defined element
$$
1+\sum_{l=0}^\infty\theta_l h_{1,l}+\frac{1}{2}(\sum_{l=0}^\infty\theta_l h_{1,l})^2+\cdots
$$
in $B_{\mathrm{HdR}}$, and we regard it as an operator on $B_{\mathrm{HdR}}$ by left multiplication. Set
$$
\delta_0:=\mathrm{exp}(\sum_{l=0}^\infty\theta_l h_{1,l})\mathrm{Shift}_0,\quad \delta_s:=\mathrm{Shift}_s, \ s>0.
$$
Finally, we define the operator $D:\mathrm{Hom}_{\mathrm{Set}}(\mathbb Z\times\mathbb Z,B_{\mathrm{HdR}})\to\mathrm{Hom}_{\mathrm{Set}}(\mathbb Z\times\mathbb Z,B_{\mathrm{HdR}})$ by sending $\varphi=\{\varphi(r,s)\}\in \mathrm{Hom}_{\mathrm{Set}}(\mathbb Z\times\mathbb Z,B_{\mathrm{HdR}})$
to
$$
(r,s)\mapsto\nabla\varphi(r,s-1)+(-1)^{s}\sum^r_{l=0}(-1)^l\delta_l\varphi(r-1,s)-\Theta\varphi(r,s)).
$$
\begin{theorem}\label{solvability}
The following two-dimensional discrete initial value problem over $B_{\mathrm{HdR}}$ is solvable:
\begin{eqnarray}\label{initial problem}
\left\{\begin{matrix}
D\varphi=0,~\varphi\in\mathrm{Hom}_{\mathrm{Set}}(\mathbb Z\times\mathbb Z,B_{\mathrm{HdR}});\\
\varphi(0,s)=\sum_{I\subset\mathbb N_{>0},|I|=s}e_I\zeta_{0,I},~s\geq0;\\
\varphi(r,s)=0,~r<0~\mathrm{or}~s<0.
\end{matrix}\right.
\end{eqnarray}
Here $\zeta_{0,\emptyset}=1$, and $\zeta_{0,I}:=\zeta_{0,i_1}\cdots\zeta_{0,i_s}$ for $I=\{i_1,\cdots,i_s\}$ with $i_1<\cdots<i_s$.
\end{theorem}
In the following, we shall give an \emph{explicit} solution to the initial value problem. There might be a proof showing the abstract existence of such a solution (via a certain obstruction theory). But it will not suffice for our purpose. This is mainly because the intersection condition is given in terms of local coordinates. We need to show our solution preserves the intersection condition. Also, that our solution can be reduced modulo $p$ is quite essential.
 \begin{itemize}
%\item[$[n]$] For $n\in \mathbb Z_{>0}$, $[n]=\{1,2,\cdots,n\}$.
\item[$\overline i$\ ] $\{i^k\}_{k\in\mathbb Z_{>0}}$ such that $i_k\geq0$ and there are only finitely many $k\in\mathbb Z_{>0}$ such that $i^k>0$.

\item[$\underline i$\ ] $\{i_l\}_{l\in\mathbb Z_{>0}}$ such that $i_l\geq0$ and there are only finitely many $l\in\mathbb Z_{>0}$ such that $i_l>0$.
\item[$\underline s$\ ] $\{s_l\}_{l\in\mathbb Z_{\geq0}}$ such that each $s_l\geq0$ and $s_l>0$ holds only for finitely many $k$.

\item[$\underline{\underline j}$\ ] It is a family $\{j_{k,l}\}_{(k,l)\in\mathbb Z_{>0}\times\mathbb Z_{>0}}$ such that $j_{k,l}\geq0$ holds for any pair $(k,l)$ and $j_{k,l}>0$ holds for only finitely many pair $(k,l)$.

\item[$\underline{\underline{j}}_{\overline{i}}$] $\prod_{k}j_{k,i^{k}}$, here $j_{k,0}:=1$.

\item[$\theta_{\overline i}$]  $\prod_{k}\theta_{i^k}$, here $\theta_0:=1$.

\item[$\theta^{\underline{\underline j}}$] $\prod_{l}\theta_{l}^{\sum_{k}j_{k,l}}$.

\item[$e_{\underline i,\overline i}$\ ] Let $\underline I:=\{i_l>0|l\in\mathbb Z_{>0}\}$ and $\overline I:=\{i^k>0|k\in\mathbb Z_{>0}\}$. Assume $\underline I=\overline I=\emptyset$. Set $e_{\underline i,\overline i}:=e_\emptyset$. Assume $\underline I\neq\emptyset,~\overline I\neq\emptyset$. Set
$$
\underline I=\{i_{l_1},\cdots,i_{l_s}\}~\mathrm{with}~l_1<\cdots<l_s,~\overline I=\{i^{k_1},\cdots,i^{k_r}\}~\mathrm{with}~k_1<\cdots<k_r.
$$
If $|\underline I\cup\overline I|<r+s$, then set $e_{\underline i,\overline i}:=0$. If $|\underline I\cup\overline I|=r+s$, then set
$$
e_{\underline i,\overline i}:=\mathrm{sgn}(\sigma)e_{\underline I\cup\overline I},
$$
where $\sigma$ is a permutation of $\{1,\cdots,r+s\}$ such that
$$
i'_{\sigma(1)}<\cdots<i'_{\sigma(r+s)},~(i'_1,\dots,i'_{r+s}):=(i_{l_1},\cdots,i_{l_s},i^{k_1},\cdots,i^{k_r}).$$
The remaining two cases $\underline I\neq\emptyset,\overline I=\emptyset$ and $\underline I=\emptyset,\overline I\neq\emptyset$ can be discussed similarly.

\item[$h^{[\underline{\underline j}]}$] $\prod_{k,l}\frac{h_{k,l}^{j_{k,l}}}{j_{k,l}!}$.

\item[$\zeta_{\underline s,\underline i}$] For $s=0$, set $\zeta_{\underline s,\underline i}:=1$. Otherwise, write the set of nonzero entries of $\underline s$ by $\{k_1,\cdots,k_q\}$ with $k_1<\cdots<k_q$, and set
$\zeta_{\underline s,\underline i}:=\zeta_{l_1,i_1}\cdots\zeta_{l_s,i_s}$, where the sequence $l_1,\cdots,l_s$ is $s_{k_1}$-times $k_1$, $\cdots$, $s_{k_q}$-times $k_q$.
\item[$a(\underline{\underline j},\underline s)$] $\frac{\prod_ls_l!}{\prod^\infty_{p=1}\prod^{s_{p-1}}_{q=0}\mathrm{max}\{1,q+\sum_{k\geq p,l}j_{k,l}+\sum_{l\geq p}s_l\}}$.
\item[$T(r)$\ ] $\{i^k\}_{k\in\mathbb Z_{>0}}$ such that $i^k>0$ for $k\leq r$ and $i^k=0$ for $k>r$.
\item[$\mathbb Z^{\underline s}_{>0,\uparrow}$\ ] Let $s:=\sum_l s_l>0$. If $s=0$, then put $\mathbb Z^{\underline s}_{>0,\uparrow}$ to be the set of single sequence $\underline 0=\{0\}_{l\in\mathbb Z_{>0}}$. Assume $s>0$.
Let $\{l|s_l>0\}=\{l_1,\cdots,l_q\}$ with $l_1<\cdots<l_q$. Put $\mathbb Z^{\underline s}_{>0,\uparrow}$ to be the set of sequences
$$
i^1_1,\cdots,i^1_{s_{l_1}},\cdots,i^q_1,\cdots,i^q_{s_{l_q}},0,\cdots,0,\cdots
$$
satisfying $i^1_1<\cdots<i^1_{s_{l_1}},\cdots,i^q_1<\cdots<i^q_{s_{l_q}}$.

\item[$T(r,s)$\ ] For any $r,s\geq0$, let $T(r,s)$ be the set consisting of triples $(\underline s,\underline i,\underline{\underline j})$ satisfying the following conditions:
\begin{itemize}
\item $j_{k,l}=s_k=0$ for $k>r$;
\item $\underline i\in\mathbb Z^{\underline s}_{>0,\uparrow}$.
\end{itemize}

\end{itemize}

\begin{proof}
Let $r,s\geq 0$, we define
\begin{eqnarray}\label{varphi construction}
\varphi_\infty(r,s):=\sum_{(\underline s,\underline i,\underline{\underline j})\in T(r,s)}[\sum_{\overline i\in T(r)}a(\underline{\underline j},\underline s)\underline{\underline j}_{\overline i}\theta^{-1}_{\overline i}\theta^{\underline{\underline j}}e_{\underline i,\overline i}]h^{[\underline{\underline j}]}\zeta_{\underline s,\underline i}.
\end{eqnarray}
The verification detail that $\varphi_\infty=\{\varphi_\infty(r,s)\}$ is a solution of \ref{initial problem} is contained in \cite{SZ}.
\end{proof}

\subsection{The $\infty$-homotopy}
 Let $\varphi_m$ be the natural reduction of $\varphi_\infty$ to $B_{\mathrm{HdR},m}$, which is viewed as an element of $\mathrm{Hom}_{\mathrm{Set}}(\mathbb Z\times\mathbb Z,B_{\mathrm{HdR}})$ using \eqref{decomposition B}. Clearly, we have
$$
\varphi_m(r,s)\to\varphi_\infty(r,s),~m\to\infty.
$$
Let $m=p$ be a prime number. By the very construction of $\varphi$ \eqref{varphi construction}, under the natural projection
$$\mathrm{Hom}_{\mathrm{Set}}(\mathbb Z\times\mathbb Z,B_{\mathrm{HdR}})\to\mathrm{Hom}_{\mathrm{Set}}(\mathbb Z\times\mathbb Z,B_{\mathrm{HdR},p}),$$
$\{\varphi_p(r,s)\}_{(r,s)}$ are contained in
$$\mathrm{Hom}_{\mathrm{Set}}(\mathbb Z\times\mathbb Z,B_{\mathrm{HdR},p}^{\mathbb Z_{(p)}}
)\subset \mathrm{Hom}_{\mathrm{Set}}(\mathbb Z\times\mathbb Z,B_{\mathrm{HdR},p}).$$
Note that the components $\nabla,\Theta,\delta_l$ defining the operator $D$ naturally pass on $B^{\mathbb Z_{(p)}}_{\mathrm{HdR},p}$, which are denoted by $\nabla_p,\Theta_p,(\delta_l)_p$ respectively. Hence $D$ induces the operator $D_p$ on $\mathrm{Hom}_{\mathrm{Set}}(\mathbb Z\times \mathbb Z,B_{\mathrm{HdR},p}^{\mathbb Z_{(p)}})$. Certainly, $D_p\varphi_p=0$. Written explicitly, it means the following equality
\begin{eqnarray}\label{reduction p}
\qquad(-1)^s\nabla_p\varphi_p(r,s)+(-1)^{s+1}\Theta_p\varphi_p(r,s+1)=\sum^r_{l=0}(-1)^l(\delta_l)_p\varphi_p(r-1,s+1).
\end{eqnarray}

Next, we shall realize \ref{reduction p} as an $\infty$-homotopy between the Higgs complex associated to $(E,\theta)$ and the Frobenius pushforward of the de Rham complex associated to its inverse Cartier transform. Suppose $\ell <p$. Let $\mathrm{Frame}$ be the frame bundle of $\Omega_{X'_{\log}/k}$, and $\sL=F_*\sL_{\sX/\sS}$. Let $U\subset X$ be an open subset and $r$ an non-negative integer. Fix
 \begin{itemize}
 \item a frame $(\omega'_1,\cdots,\omega'_n)\in\mathrm{Frame}(U)$;
 \item an $r+1$-tuple $(\tilde F_0,\cdots,\tilde F_r)$, where $\tilde F_i$ is a logarithmic Frobenius lifting over $\tilde U$;
 \item a section of $\tau_{<p-\ell}\Omega^\bullet(E,\theta)$ over $U'$ given by (where $e,e_{i_1,\cdots,i_q}\in E(U')$)
 $$
 e+\sum_{0<q<p-\ell}\frac{1}{q!}\sum_{1\leq i_1,\cdots,i_q\leq n}e_{i_1,\cdots,i_q}\otimes\omega_{i_1}'\wedge\cdots\wedge\omega_{i_q}'.
 $$
 \end{itemize}
We define a ring homomorphism
 $$
 B^{\mathbb Z_{(p)}}_{\mathrm{HdR},f}\to\Gamma(U',R(E,\theta)\otimes_{\sO_{X'}} F_*\Omega^\bullet_{X_{\log}/k})$$
 as follows:
 \begin{itemize}
 \item $\theta_i\mapsto\vartheta_i\otimes 1$ for $i\leq n$ and $\theta_i\mapsto0$ for $i>n$, where $\theta=\sum_{i=1}^n\vartheta_i\otimes\omega_i'$;
 \item $e_\emptyset\mapsto e\otimes 1$, $e_I\mapsto e_{i_1,\cdots,i_q}\otimes 1$ for $I=\{i_1,\cdots,i_q\}$ with $i_1<\cdots<i_q\leq n$ and $e_I\mapsto 0$ otherwise;
 \item $h_{k,l}\mapsto 1\otimes h_{\tilde F_{k-1}\tilde F_{k}}(\omega_l)$ for $k\leq r$ and $l\leq n$ and $0$ otherwise;
 \item $\zeta_{k,l}\mapsto1\otimes \zeta_{\tilde F_k}(\omega_l)$ for $k\leq r$ and $l\leq n$, and $0$ otherwise.
 \end{itemize}
Regard $B^{\mathbb Z_{(p)}}_{\mathrm{HdR},f}$ as a constant sheaf on $X'$. Then the above construction gives rise to a morphism of sheaves on $X'$:
  $$\mathrm{Frame}\times\bigsqcup_{r\geq0}\mathcal L^{r+1}\times\tau_{<p-\ell}\Omega^\bullet(E,\theta)\to\mathcal Hom(B^{\mathbb Z_{(p)}}_{\mathrm{HdR},f},R(E,\theta)\otimes_{\sO_{X'}} F_*\Omega^\bullet_{X_{\log}/k}).
  $$
We shall take the component of the previous morphism along the projection
$$
 R(E,\theta)\otimes_{\sO_{X'}} F_*\Omega^\bullet_{X_{\log}/k}\to E\otimes_{\sO_{X'}}  F_*\Omega^\bullet_{X_{\log}/k}.
$$
Now return to the setting of \S3.2. Recall that $(H,\nabla)=C^{-1}(E,\theta)$, the inverse Cartier transform of $(E,\theta)$. There are natural identifications
$$
E_{U'}\otimes F_*\Omega^\bullet_{U_{\log}/k}\cong F_*\Omega^\bullet(H_{\tilde F_0},\nabla_{\tilde F_0})\cong F_*\Omega^\bullet(H,\nabla)_U.
$$
So the above discussion provides us the following morphism
$$
\mathrm{Frame}\times\bigsqcup_{r\geq0}\mathcal L^{r+1}\times\tau_{<p-\ell}\Omega^\bullet(E,\theta)\to\mathcal Hom(B^{\mathbb Z_{(p)}}_{\mathrm{HdR},f},F_*\Omega^\bullet(H,\nabla)),
$$
which induces the following evaluation morphism
$$
\mathrm{ev}: \mathrm{Frame}\times\bigsqcup_{r\geq0}\mathcal L^{r+1}\times B^{\mathbb Z_{(p)}}_{\mathrm{HdR},f}\to \mathcal Hom_{\mathcal O_{X'}}(\tau_{<p-\ell}\Omega^\bullet(E,\theta),F_*\Omega^\bullet(H,\nabla)).
$$

\begin{lemma}\label{indep frame}
For $r,s\geq 0$, the induced morphism $\mathrm{ev}(-,-,\varphi_p(r,s))$
$$
\mathrm{Frame}\times\mathcal L^{r+1}\to\mathcal Hom_{\mathcal O_{X'}}(\tau_{<p-\ell}\Omega^{r+s}(E,\theta),\tau_{<p-\ell}F_*\Omega^s(H,\nabla)),
$$
is independent of the first factor $\mathrm{Frame}$. Consequently, we obtain a well-defined morphism
$$
\mathrm{ev}(-,\varphi_p(r,s)):\mathcal L^{r+1}\to\mathcal Hom_{\mathcal O_{X'}}(\tau_{<p-\ell}\Omega^{r+s}(E,\theta),\tau_{<p-\ell}F_*\Omega^s(H,\nabla)).
$$
 \end{lemma}
 \begin{proof}
 This lemma follows from the strong symmetry of $\varphi_p(r,s)$, see \cite{SZ} for details.
 \end{proof}

\begin{theorem}\label{appendix construction Ho}
Let $\mathcal X/\mathcal S, (E,\theta), (H,\nabla),\sL$ be as in Theorem \ref{dec thm}. Then $\{\mathrm{Ho}^r\}_{r\geq 0}$ with
$$
\mathrm{Ho}^r=\mathrm{ev}(-,\varphi_p(r,-)):\mathcal L^{r+1}\to\mathcal Hom^{-r}(\tau_{<p-\ell}\Omega^*(E,\theta),\tau_{<p-\ell}F_*\Omega^*(H,\nabla))
$$
defines an $\mathcal L$-indexed $\infty$-homotopy from $\tau_{<p-\ell}\Omega^*(E,\theta)$ to $\tau_{<p-\ell}F_*\Omega^*(H,\nabla)$.
\end{theorem}
\begin{proof}
The theorem is nothing but \eqref{reduction p} under the evaluation morphism. However, we should take care of the nilpotent condition on $\theta$. Applying $\mathrm{Ho}^r(\tilde F_0,\cdots,\tilde F_r;-)$ to both sides of \eqref{reduction p}   yields
\begin{eqnarray}\label{LHS infty}
\quad\mathrm{ev}(\tilde F_0,\cdots,\tilde F_r;(-1)^s\nabla_p\varphi_p(r,s))+\mathrm{ev}(\tilde F_0,\cdots,\tilde F_r;(-1)^{s+1}\Theta_p\varphi_p(r,s+1))\end{eqnarray}
and
\begin{eqnarray}\label{RHS infty}
\mathrm{ev}(\tilde F_0,\cdots,\tilde F_r;\sum^r_{k=0}(-1)^k(\delta_k)_p\varphi_p(r-1,s+1)),
\end{eqnarray}
respectively. Claim that
\begin{eqnarray}\label{varphi nabla evaluating}
\mathrm{ev}(\tilde F_0,\cdots,\tilde F_r;(-1)^s\nabla_p\varphi_p(r,s))=(-1)^s\nabla\mathrm{ev}(\tilde F_0,\cdots,\tilde F_r;\varphi_p(r,s)).
\end{eqnarray}
Indeed, there is an obvious equality
$$
\mathrm{ev}(\tilde F_0,\cdots,\tilde F_r;(-1)^s\nabla\varphi_p(r,s))=(-1)^s\nabla\mathrm{ev}(\tilde F_0,\cdots,\tilde F_r;\varphi_p(r,s)).
$$
Since the difference $\nabla\varphi_p(r,s)-\nabla_p\varphi_p(r,s)$ belongs to $\mathfrak b_p$, it can be expressed as an $\mathbb F_p$-linear formal combination of elements of form
$$\underline{\underline j}_{\overline i}\theta^{-1}_{\overline i}\theta^{\underline{\underline j}}\theta_le_{\underline i,\overline i}h^{[\underline{\underline j}]}\zeta_{0,l}\zeta_{\underline s,\underline i},~\overline i\in T(r),~l\geq1,~\sum_{k,l} j_{k,l}+\sum_k s_k=p-1.$$
Note the total power of $\theta_l$ in any element above is not less than $p-(r+s)$. Therefore, as $(E,\theta)$ is nilpotent of level $\leq \ell$, we have for $r+s<p-\ell$
$$
\mathrm{ev}(\tilde F_0,\cdots,\tilde F_r;(-1)^s(\nabla-\nabla_p)\varphi_p(r,s))=0.
$$
The claim follows. Clearly, one has
$$
\mathrm{ev}(\tilde F_0,\cdots,\tilde F_r;(-1)^{s+1}\Theta_p\varphi_p(r,s+1))=(-1)^{s+1}\mathrm{ev}(\tilde F_0,\cdots,\tilde F_r;\varphi_p(r,s+1))\theta.
$$
This shows that \eqref{LHS infty} equals
$$
(d_{\mathcal Hom^*}\mathrm{Ho}^r)_{s+1}(\tilde F_0,\cdots,\tilde F_r),
$$
where $(d_{\mathcal Hom^*}\mathrm{Ho}^r)_{s+1}$ is one of the components of $d_{\mathcal Hom^*}\mathrm{Ho}^r$. By a similar argument, \eqref{RHS infty} becomes
$$
(\delta\mathrm{Ho}^{r-1})_{s+1}(\tilde F_0,\cdots,\tilde F_r),
$$
where
$$(\delta\mathrm{Ho}^{r-1})_{s+1}:\mathcal L^r\to\mathcal Hom_{\mathcal O_{X'}}(\tau_{<p-\ell}\Omega^{r+s}(E,\theta),\tau_{<p-\ell}F_*\Omega^{s+1}(H,\nabla))$$
is one of the components of $\delta\mathrm{Ho}^{r-1}$.
Consequently, one has
$$
d_{\mathcal Hom^*}\mathrm{Ho}^r=\delta\mathrm{Ho}^{r-1},$$
from which this theorem follows.
 \end{proof}
This completes the first half of Theorem \ref{dec thm}. To handle the remaining part, we need to show further that the so-constructed $\infty$-homotopy preserves the intersection condition. Clearly, the problem now is a local one. So we may assume that $X$ admits an admissible system of local coordinates $t=(t_1,\cdots,t_n)$. It suffices to show the following
\begin{claim}\label{truncation ev}
$\mathrm{ev}((\tilde F_0,\cdots,\tilde F_r),\varphi_p(r,s))$ in Lemma \ref{indep frame} restricts to a morphism
\begin{eqnarray}
\tau_{<p-\ell}\Omega_{int}^{r+s}(E,\theta)\to\tau_{<p-\ell}F_*\Omega_{int}^s(H,\nabla).
\end{eqnarray}
\end{claim}
In the following, we shall fix the identification
$$
(H,\nabla)=(H_{\tilde{F}_0},\nabla_{\tilde{F}_0}).
$$
Via the local coordinate $t$, we write
$$
\nabla=\sum_{i=1}^n\nabla_i\otimes \omega_i,
$$
where $\omega_i=d\log t_i$, as we did in \S2.1. Recall that for $w\in \N^n$, we defined the submodule of $H$
$$
H_{w}=\sum_{\alpha+\beta=w,\alpha,\beta\in\mathbb N^n}t^\alpha\nabla^\beta H.
$$
\begin{lemma}\label{H_w general}
For any $w=(w_1,\cdots,w_n)\in\{0,1\}^n$, we have
\begin{eqnarray}\label{theta=int}
H_{w}=\sum_{\alpha+\beta=w,\alpha,\beta\in\mathbb N^n}t^\alpha(F^*\theta)^\beta H,
\end{eqnarray}
where $(F^*\theta)^\beta=\prod_{i=1}^n(F^*\theta_i)^{\beta_i}$ with $\beta=(\beta_1,\cdots,\beta_n)$.
	\end{lemma}
		
\begin{proof}
Set
$$
H^w:=\sum_{\alpha+\beta=w,\alpha,\beta\in\mathbb N^n}t^\alpha(F^*\theta)^\beta H.
$$
We do induction on $|w|=\sum_{i=1}^nw_i$. The lemma holds trivially for $|w|=0$, i.e. $w=0$. We assume the truth for $|w|=N-1$. Without loss of generality, we take $w\in\{0,1\}^n$ with
$$
w=(1,0,\cdots,0)+w', \quad |w'|=N-1.
$$
For each $i$, there are $\{g_{ij}\}\subset \Gamma(X,\mathcal O_X)$ such that
$$
\nabla_i=F^*\theta_i+(\nabla_{\mathrm{can},i}+t_i\sum_{j=1}^ng_{ij}F^*\theta_j).
$$
where $\nabla_{\mathrm{can},i}(\lambda F^*e)=t_i\frac{\partial \lambda}{\partial t_i}F^*e$ for $\lambda\in\mathcal O_X,e\in E$.
Note that  for $s\neq i$, we have the following relations in $\mathrm{End}_k(H)$:
$$
[\nabla_{\mathrm{can},s},t_i]=[\nabla_{\mathrm{can},s},F^*\theta_i]=0.$$
Then it follows that
$$
\begin{array}{rcl}
H^w&\subset&t_1H_{w'}
+(\nabla_1H_{w'}-\nabla_{\mathrm{can},1}H^{w'}-(t_1\sum_{j=1}^ng_{1j}F^*\theta_j)H^{w'})\\
&\subset& H_w+t_1H^{w'}
= H_w+t_1H_{w'}= H_w.
\end{array}
$$
By symmetry, $H_w\subset H^w$. The lemma follows.
\end{proof}	
\begin{corollary}\label{infty homotopy intersection}
Notations as in Theorem \ref{appendix construction Ho}. Then $\{Ho^r\}_{r\geq 0}$ restricts to an $\mathcal L$-indexed $\infty$-homotopy from $\tau_{<p-\ell}\Omega_{int}^*(E,\theta)$ to $\tau_{<p-\ell}F_*\Omega_{int}^*(H,\nabla)$.
\end{corollary}
\begin{proof}
In order to show claim \eqref{truncation ev}, it suffices to consider the case of $r+s<p$ and $s>0$. Note that $\mathrm{ev}((\tilde F_0,\cdots,\tilde F_r),\varphi_p(r,s))$ always extends(=without truncation) to a morphism
$$
\mathrm{ev}((\tilde F_0,\cdots,\tilde F_r),\varphi_p(r,s)):\Omega^{r+s}(E,\theta)\to F_*\Omega^s(H,\nabla).
$$
By Lemma \ref{H_w general}, it follows that for $\beta=\omega_I\in P^{r+s}$, the evaluation morphism sends any section of $E_{w(\beta)}\otimes\beta$ to a section of
$$
\sum H_{w(\beta)}\otimes\zeta_{\tilde F_{l_1}}(\omega_{i_1})\wedge\cdots\zeta_{\tilde F_{l_s}}(\omega_{i_s}),
$$
where the sum takes over $0\leq l_1,\cdots,l_s\leq r,i_1,\cdots,i_s\in I$. Clearly, the latter subsheaf is contained in $F_*\Omega^s_{int}(H,\nabla)$. Hence \eqref{truncation ev} follows from Proposition \ref{multiweight Higgs}.
\end{proof}

\end{document}